\documentclass[11pt]{article}

\usepackage[nohead,margin=1in]{geometry}
\usepackage{amssymb, amsmath, amsthm,amsfonts}
\usepackage{amsmath}
\usepackage{graphicx,epsfig}
\usepackage[tight]{subfigure}
\usepackage{cite}
\usepackage{times}
\usepackage{float}
\graphicspath{{./pics/}}
\usepackage{color}
\usepackage{epstopdf}
\usepackage{mathrsfs}
\usepackage{mathtools}
\usepackage{threeparttable}
\usepackage{algorithm,algorithmic}
\usepackage{booktabs}
\usepackage{verbatim}
\usepackage{bm}

\usepackage[colorlinks,linktocpage,linkcolor=blue]{hyperref}

\numberwithin{equation}{section}
\usepackage{array}
\newtheorem{thm}{Theorem}[section]
\newtheorem{lem}{Lemma}[section]
\newtheorem{remark}{Remark}[section]

\newtheorem{coro}{Corollary}[section]
\newtheorem{assum}{Assumption}[section]

\allowdisplaybreaks

\renewcommand{\d}{\mathrm{d}}

\def\F{\mathcal{F}}
\def\G{\mathcal{G}}
\def\e{{\rm e}}
\newcommand{\Tau}{\bm{\tau}}

\title{Exponential Low-Regularity Parareal Algorithms for  \\Nonlinear Schr\"{o}dinger Equations\thanks{This work is partially supported by National Natural Science Foundation of China (Project 12422117 and Project 12426312),
Hong Kong Research Grants Council (Project 15302323) and an internal grant of Hong Kong Polytechnic University (Project P0053938, Work Programme: 4-ZZVA).}}
\author{ 
Qingle Lin\thanks{Department of Applied Mathematics, The Hong Kong Polytechnic University, Kowloon, Hong Kong, P.R. China (\texttt{qingle.lin@connect.polyu.hk, zhizhou@polyu.edu.hk})}
\and Zhi Zhou\footnotemark[2]
}

\begin{document}

\maketitle
 
\begin{abstract}
The parareal algorithm is one of the most widely studied parallel-in-time methods for the numerical approximation of time-dependent problems. For non-diffusive equations, however, standard parareal methods may converge slowly or even become unstable due to the absence of damping, while nonlinear interactions can transfer and amplify phase errors across Fourier modes. In this work, we consider the nonlinear Schr\"odinger equation (NLS) as a representative non-diffusive model and analyze parareal algorithms with an exact fine propagator, with particular emphasis on the design of suitable coarse propagators. We establish a general convergence framework, valid for solutions with limited regularity, under stability and local truncation error assumptions on the coarse propagator. These assumptions are verified for selected exponential low-regularity integrators designed for one-dimensional quadratic and cubic NLS equations, which achieve optimal approximation orders without derivative loss. To the best of our knowledge, this is the first construction of parareal algorithms for NLS equations that are provably linearly convergent, with a contraction factor proportional to the coarse time-step size even for solutions of limited regularity. Numerical experiments on quadratic, cubic, and quintic NLS equations demonstrate rapid convergence and improved performance over parareal variants using classical coarse propagators, including Lie and Strang splitting methods and first- and third-order exponential Runge--Kutta integrators.


\end{abstract}

\section{Introduction}
In this paper, we consider the parallel-in-time (PinT) numerical approximation of the nonlinear Schr\"{o}dinger equation~(NLS) 
\begin{equation}\label{eqn:NLSE}
    i\partial_{t} u=-\Delta u+F(u,\overline{u}),\quad~(t,x) \in (0,T)\times\mathbb{T}^d,
\end{equation}
subject to the initial condition
$u(0)=u_0 \in H^r(\mathbb{T}^d)$. Here, $\Delta=\sum_{i=1}^{d} \partial_{x_i}^2$ and $\mathbb{T}^d=(0,2\pi)^d$ denotes the $d$-dimensional torus, corresponding to periodic boundary conditions.
NLS-type models arise in many areas of physics, including nonlinear optics, Bose--Einstein condensates, deep water waves, and plasma physics~\cite{hasegawa1973transmission,gross1961structure,pitaevskii1961vortex,PhysRevLett.106.204502,BERGE1998259}.
The numerical approximation of~\eqref{eqn:NLSE} is typically carried out using sequential time-stepping schemes, which advance the solution incrementally in time. This inherently sequential structure constitutes a major computational bottleneck, particularly for long-time simulations. Moreover, as processor clock speeds have approached fundamental physical limits, further improvements in computational performance must increasingly rely on greater parallelism through the use of larger numbers of cores~\cite{MR4926314}.

PinT methods have been an active area of research for more than two decades in the numerical solution of large-scale evolution problems, particularly when spatial parallelism becomes saturated. The earliest pioneering work can be traced back to Nievergelt in 1964~\cite{MR176617}. Several well-established PinT algorithms, such as the parareal method~\cite{MR1842465,MR2306258}, multigrid reduction in time (MGRIT)~\cite{MR3716570,MR3499068}, parallel full approximation scheme in space and time (PFASST)~\cite{MR2979518,MR3504550}, diagonalization technique~\cite{MR4009694,MR2954747,WuZhouZhou:2022}, and Laplace transform~\cite{MR1648403,MR1975267} have proven highly effective, with rigorous error estimates. For a broader overview, we refer readers to the comprehensive review in \cite{Gander:2015,OngSchroder:2020,MR4926314} and references
therein. These algorithms have achieved considerable success when applied to ordinary differential equations and diffusive partial differential equations.

In the pioneering work of Lions, Maday, and Turinici~\cite{MR1842465}, the parareal algorithm was introduced as a PinT method for evolution problems. The method partitions the time interval into subintervals, initializes the solution with a computationally cheap coarse propagator (CP), and then iteratively corrects it in parallel on each subinterval using a more accurate but more expensive fine propagator (FP). In this framework, the CP provides a low-resolution approximation, while the FP delivers high-resolution accuracy. Since its introduction, the parareal algorithm has been successfully applied to a wide range of problems, including fluid dynamics~\cite{MR2235770,MR3617079}, molecular dynamics~\cite{MR4369575,audouze:hal-00358459}, optimal control~\cite{MR2639234,MR4150714}, and stochastic models~\cite{MR4049401,MR4879400}, among others.

Despite these successes, the standard parareal algorithm often performs poorly for non-diffusive problems, where the lack of natural damping presents a fundamental challenge for many parallel-in-time (PinT) algorithms \cite{MR4926314,MR3989875}. For linear problems, classical Runge-Kutta methods used as the CPs lead to slow convergence or even divergence~\cite{10.1007/978-3-030-75933-9_5,MR2306258}. By contrast, exponential-type integrators can resolve the phase discrepancy between the CP and FP exactly, thereby enhancing the stability of parareal \cite{MR3817021,MR3033094}. For nonlinear problems, the accuracy of the correction term at each parareal iteration depends on the regularity of the solution from the previous iteration~\cite{MR3033060,bal2005convergence}. This dependence, together with the lack of a smoothing effect, leads to fundamental challenges in achieving rapid and robust convergence. As an illustrative example, in \cite{linear_conv}, we analyze the convergence of parareal algorithms for semilinear parabolic equations. A key step in the convergence analysis is to prove the following estimate, 
for some $r \in [0,2]$, 
\begin{equation*}
    \Big\| \sum_{j=0}^{N_{c}-1} \e^{-j\Tau A}\big( \left( \mathcal{F}_{\Tau}-\mathcal{G}_{\Tau} \right) \left( u_j \right) -\left( \mathcal{F}_{\Tau}-\mathcal{G}_{\Tau}\right) \left( v_j\right) \big)\Big\|_{H^{r}} \leq C\Tau \max_{0\leq j\leq N_{c}} \|u_j-v_j\|_{H^{r}},  
\end{equation*}
where $A=-\Delta$, $\mathcal{F}_{\Tau}$ and $\mathcal{G}_{\Tau}$ denote the fine and coarse propagators, respectively, $\Tau$ is the coarse time-step size, $N_c$ is the number of coarse intervals, and $C$ is a constant independent of $\Tau$ and $N_c$. An important feature of this estimate is that the norms on both sides are taken in the same space, which is crucial for the subsequent analysis.
This estimate is crucially based on the parabolic smoothing property $\| A\e^{-j\Tau A}\|_{L(H^{r})} \leq \left( j\Tau \right)^{-1}$. However, this property fails when $A = i\Delta$ in~\eqref{eqn:NLSE}, since the solution operator $\e^{i\Tau\Delta }$ is unitary and hence has no smoothing effect. 
Therefore, one cannot expect the above estimate to hold for standard integrators in this setting. Dai and Maday also identified the loss of regularity as the central obstacle in the analysis of parareal algorithm \cite{MR3033060}. This observation motivates the use of numerical schemes that attain the optimal approximation order without derivative loss.  More recently, Buvoli and Minion revisited non-diffusive equations by repartitioning the problem and employing exponential integrators for both the CP and FP~\cite{MR4668796}. However, their analysis is restricted to the Dahlquist equation, and their fast convergence requires high regularity because of the repartitioning strategy.

In this work, we develop and analyze parareal algorithms for the nonlinear Schr\"odinger equation (NLS), assuming that the fine propagator is exact, and investigate suitable choices of CPs. In particular, we rigorously show that the parareal exhibits fast convergence behavior, under some stability and local truncation error assumptions on the coarse propagator, allowing for solutions with limited regularity.  In particular, under some conditions on the CP (in Assumption \ref{assum:cp}),  we prove that (in Theorem \ref{thm:general_conv})
\begin{equation}\label{eqn:error-main}
 \max_{1 \le n\le N_c}  \|U_{n}^{k}-U_n \|_{H^r} \leq  C_1 (C_2 \Tau )^{k},\quad r>d/2.
\end{equation}
Here, $U_n=u(T_n)$ denotes the exact solution of the NLS~\eqref{eqn:NLSE} at time level $T_n=n\Tau$, while $U_n^k$ denotes the numerical solution produced by the parareal algorithm at the $k$-th iteration and at time level $T_n$.
This estimate shows that the error decreases geometrically with respect to the parareal iteration index $k$, provided that $\Tau$ is sufficiently small. Moreover, we verify those conditions in Assumption \ref{assum:cp} for certain exponential low-regularity integrators specifically designed for one-dimensional quadratic and cubic NLS equations, which are capable of attaining the optimal approximation order without derivative loss. With these theoretically grounded choices of coarse propagators, the resulting exponential low-regularity parareal algorithm converges linearly, with a contraction factor proportional to the coarse time-step size $\Tau$, as indicated in the error estimate \eqref{eqn:error-main}. To the best of our knowledge, this is the first work to construct parareal algorithms for NLS equations that are provably linearly convergent, achieved through suitably designed exponential low-regularity coarse propagators. Finally, numerical experiments on the quadratic, cubic, and quintic Schr\"{o}dinger equations demonstrate the rapid convergence of the proposed exponential low-regularity parareal algorithm, compared with several widely used classical solvers, including Lie splitting, Strang splitting, and first- and third-order exponential Runge--Kutta methods.

The development of exponential low-regularity integrators~(ELRIs) originated with~\cite{MR3807360} and has since led to a variety of numerical schemes for dispersive and hyperbolic problems. In contrast to time-stepping schemes based on Taylor expansions and classical exponential integrators, ELRIs are built upon the Duhamel formula and a careful, equation-specific treatment of the nonlinear term, thereby avoiding the need for high-order derivatives of the solution.
Such schemes have been developed for the NLS~\cite{MR3807360,MR4385374}, the Korteweg--de Vries equation~\cite{wu2022optimal,li2025unfiltered,li2021convergence}, the nonlinear Klein--Gordon equation~\cite{wang2022symmetric}, and the nonlinear Dirac equation~\cite{schratz2021low}, among others.  An important subclass consists of ELRIs without loss of regularity, for which the error in $H^r$ can be bounded solely in terms of the $H^r$ norm of the initial datum~\cite{li2021convergence,MR4405493}. This property plays a key role in the analysis of the present paper. These schemes rely on a refined frequency-by-frequency analysis of the nonlinear interactions.

The rest of the paper is structured as follows. Section~\ref{sec:prelim} reviews the parareal algorithm and the exponential low-regularity integrators. In Section~\ref{sect:A motivating example}, we present a motivating example of the exponential low-regularity parareal algorithm applied to the quadratic nonlinear Schr\"{o}dinger equation on $\mathbb{T}$ and derive the corresponding error estimate for the corresponding parareal algorithm. In Section~\ref{sec:general case}, we introduce a general convergence framework for the exponential low-regularity parareal algorithm and verify the required assumptions for both quadratic and cubic Schr\"{o}dinger equations.
Finally, in Section~\ref{sec:numerical}, we present the robust performance of the proposed exponential low-regularity parareal algorithm for quadratic, cubic, and quintic Schr\"{o}dinger equations on $\mathbb{T}$, as well as for the cubic Schr\"{o}dinger equation on $\mathbb{T}^2$, and compare with some other popular schemes for NLS. 
Throughout the paper, the symbols $c$ and $C$ denote generic constants that may change from line to line, and $c\pm$ denotes $c\pm\epsilon$ for an arbitrarily small $\epsilon>0$.

\section{Preliminaries}\label{sec:prelim} In this section, we briefly introduce the parareal algorithm and the basic idea underlying the exponential low-regularity integrator for the NLS.

\subsection{Parareal method}
We begin with the introduction of the parareal algorithm, one of the most popular parallel-in-time algorithms. To discretize problem~\eqref{eqn:NLSE} in time, we divide the time interval $(0,T)$ into $N$ equal subintervals, each of length $\tau = T/N$. Let $\Tau = J\tau$ $(J\in \mathbb{N})$ be the coarse time-step size, $N_c = T/\Tau \in \mathbb{N}$ and $T_n = n \Tau$. 

Let $\F$ and $\G$ be two one-step methods that are respectively called the fine and coarse propagators. In practice, the CP $\mathcal{G}$ is often an inexpensive low-order  method, whereas the FP $\mathcal{F}$ is a high-order but expensive time integrator. Given any $v\in L^2(\mathbb{T}^d)$, the CP, denoted by $\mathcal{G}_{\Tau}(v)$, evolves the initial state $v$ with time $\Tau$, and the FP is denoted by $\mathcal{F}_{\Tau}(v)$. The sequential fine solution of~\eqref{eqn:NLSE} is given by
\begin{equation}\label{eqn:fp}
    U_{n+1} = \F_{\Tau}(U_n),\quad n=0,\dotsc,N_c -1, 
\end{equation}
starting from $U_0 = u_0$. The parareal method approximates the fine solution~\eqref{eqn:fp} via the following iteration:
\begin{equation}\label{eqn:parareal}
\left\{
\begin{aligned}
U_{n+1}^{0} &= \G_{\Tau}\bigl( U_{n}^{0} \bigr), 
&& n=0,\dots,N_{c}-1, \\[6pt]
U_{n+1}^{k+1} &= \G_{\Tau}\bigl( U_{n}^{k+1} \bigr) + \bigl( \F_{\Tau} - \G_{\Tau} \bigr) \bigl( U_{n}^{k} \bigr), 
&& n=0,\dots,N_{c}-1,\quad k=0,\dots,K.
\end{aligned}
\right.
\end{equation}
The algorithm proceeds as follows. First, the initial iteration $  k=0  $ is obtained by applying the inexpensive coarse propagator sequentially across all coarse intervals. In each subsequent iteration $  k+1  $, the correction terms $  ( \mathcal{F}_{\Tau} - \mathcal{G}_{\Tau} ) (U_n^k)  $ for $  n=0,\dotsc,N_c-1  $ are evaluated using the values from the previous iteration; these fine propagations over the coarse intervals can be performed independently and thus in parallel. The corrected values $  U_{n+1}^{k+1}  $ are then updated sequentially using the coarse propagator. In practice, the sequential fine solution~\eqref{eqn:fp} provides high temporal resolution. For convenience, we assume that the FP is exact, i.e.,
\begin{equation}\label{eqn:fp_}
    U_{n+1}=\F_{\Tau}\left( U_{n} \right) =U_{n}(\Tau)=\e^{i\Tau \Delta}U_{n}-i \int_{0}^{\Tau} \e^{i\left( \Tau -s \right) \Delta} F(U_n(s),\overline{U_n} (s)) \, \d s,
\end{equation}
where $U_n (s)$ denotes the exact solution of~\eqref{eqn:NLSE} with initial value $U_n$ at time $s$.

\subsection{Exponential low-regularity integrators}
Now we briefly describe the construction of the exponential low-regularity integrators for the NLS. In contrast to time-stepping schemes based on Taylor expansions and classical exponential integrators, ELRIs are built upon the Duhamel formula and a careful, equation-specific treatment of the nonlinear term, thereby avoiding the need for high-order derivatives of the solution.

In particular, applying Duhamel's formula to \eqref{eqn:NLSE}, the exact solution satisfies 
\begin{equation*}
u\left( T_{n+1} \right) =\e^{i\Tau \Delta}u\left( T_{n} \right) -i \int_{0}^{\Tau} \e^{i\left( \Tau -s \right) \Delta } F(u(T_n+s),\overline{u}(T_n+s))\, \d s.
\end{equation*}
In contrast to classical exponential integrators~\cite{MR2652783}, which approximate $  u(T_n + s) \approx u(T_n)$, the present scheme employs the {twisted} approximation $ u(T_n + s) \approx \e^{i s \Delta} u(T_n)$ and becomes
\begin{equation}\label{eqn:LRI1}
u\left( T_{n+1} \right) \approx \e^{i\Tau \Delta}u\left( T_{n} \right) -i \int_{0}^{\Tau} \e^{i\left( \Tau -s \right) \Delta } F(\e^{is\Delta}u(T_n),\e^{-is\Delta}\overline{u}(T_n))\, \d s. 
\end{equation}
This approximation is without loss of regularity for $r> d/2$ and $0\leq s \leq \Tau$, since 
\begin{equation*}
\begin{gathered}\| u\left( T_{n}+s \right) -\e^{is\Delta }u\left( T_{n} \right) \|_{H^{r}} \leq  \int_{0}^{s} \|F(u(T_n +\xi),\overline{u}(T_n +\xi)) \|_{H^{r}} \, \d \xi\\ \leq C\int_{0}^{s} \| u\left( T_{n}+\xi \right) \|_{H^{r}}^{q} \d\xi \leq C s \sup_{0\leq \xi \leq s} \| u\left( T_{n}+\xi \right) \|_{H^{r}}^{q},\end{gathered}
\end{equation*}
where $q$ denotes the polynomial degree of $F$. Several approaches exist for approximately evaluating this integral in~\eqref{eqn:LRI1}, leading to different exponential low-regularity integrators. A general framework is presented in~\cite{MR3807360,MR4275500}, which treats the dominant frequency exactly when $F(u,\bar{u})=|u|^{2p}u$, and this gives 
\begin{equation}\label{eqn:LRI_KSRN}
    u^{n+1}=\e^{i\Tau \Delta}\big( u^{n}-i\mu \left( u^{n} \right)^{p+1} \left( \varphi_{1} \left( -2i\Tau \Delta \right) ( \overline{u^{n}} )^{p} \right) \big),~ \text{with}~\varphi_{1} \left( z \right) ={(\e^{z}-1)}/{z}.
\end{equation}
Meanwhile, improved treatments are available in some cases~\cite{MR4385374,MR4312402,MR4730242,MR4634686,MR4405493}. 

Throughout the paper, we frequently use the following  estimate: for $u,v\in H^{r}(\mathbb{T}^d)$ and $r>d/2$,
\begin{equation}\label{eqn:bilinear}
    \|uv\|_{H^r} \leq C\|u\|_{H^r} \|v\|_{H^r}.
\end{equation}

\section{Motivating example: quadratic NLS with $F(u,\overline{u})=\mu u^2$}\label{sect:A motivating example}
In this section, we consider a special case of the general model \eqref{eqn:NLSE}, namely $F(u,\overline{u})=\mu u^2$ with $d=1$, as a motivating example for the construction of the exponential low-regularity parareal algorithm. This choice leads to a quadratic nonlinear Schr\"{o}dinger equation in one spatial dimension~(see, e.g., \cite{MR2204680})
\begin{equation}\label{eqn:q_NLSE_1}
    i \partial_t u = -\partial_{x}^2 u + \mu u^2,\quad (t,x)\in (0,T)\times \mathbb{T},
\end{equation}
with the initial value {$u(0,x)=u_0(x) \in H^{\frac12+} (\mathbb{T})$}. Let $f\in L^2 (\mathbb{T})$. We write its Fourier series as $f(x) = \sum_{k\in \mathbb{Z}} \widehat{f}_k \e^{ikx}$. We define a regularized inverse derivative $\partial_x^{-1}$ via its action on Fourier coefficients:
\begin{equation*}
    \left( \partial_{x}^{-1} \right)_{k} =\begin{cases}\left( ik \right)^{-1} ,& k\neq 0,\\ 0,&  k=0,\end{cases}
\end{equation*}
such that $\partial_{x}^{-1} f\left( x \right) =\sum_{k\in \mathbb{Z} \backslash \left\{ 0 \right\}} \left( ik \right)^{-1} \widehat{f}_{k} \e^{ikx}.$ We adopt the exponential low-regularity integrator (ELRI) from~\cite[Eq.~(35)]{MR3807360} as the CP in the parareal~\eqref{eqn:parareal}: 
\begin{equation}\label{eqn:LRI_q1}
    u^{n+1}=\left( 1-2i\mu \Tau \hat{u}_{0}^{n} \right) \e^{i\Tau \partial_{x}^{2}}u^{n}+i\mu \Tau \left( \hat{u}_{0}^{n} \right)^{2} +\frac{\mu}{2} \left( \e^{i\Tau \partial_{x}^{2}}\partial_{x}^{-1} u^{n} \right)^{2}-\frac{\mu}{2} \e^{i\Tau \partial_{x}^{2}}\left( \partial_{x}^{-1} u^{n} \right)^{2} .
\end{equation}
The derivation of this scheme follows. The exact solution of~\eqref{eqn:q_NLSE_1} satisfies the Duhamel formula
\begin{equation*}
    u(T_{n+1}) = \e^{i\Tau \partial_x^2} u(T_n) - i\mu \int_{0}^{\Tau} \e^{i(\Tau -s)\partial_x^2} u(T_n + s)^2\d s.
\end{equation*}
Approximating $u(T_n + s)$ with $\e^{is \partial_x^2} u(T_n)$ yields
\begin{equation}\label{eqn:fp_quadratic}
u(T_{n+1}) = \e^{i\Tau \partial_x^2} u(T_n) - i\mu \int_{0}^{\Tau} \e^{i(\Tau -s)\partial_x^2} (\e^{is \partial_x^2} u(T_n))^2\d s -i\mu {\rm R}_{1}\left( \Tau ;u\left( T_{n} \right) \right),
\end{equation}
where the remainder ${\rm R}_1$ is defined as
\begin{equation*}
{\rm{R}}_{1}(s; v_{0}) = \int_{0}^{s} \mathrm{e}^{i(s - \xi) \partial_{x}^{2}} \bigl( v(\xi)^{2} - \bigl( \mathrm{e}^{i\xi \partial_{x}^{2}} v_{0} \bigr)^{2} \bigr)\, \d \xi.
\end{equation*}
where $v(\xi)$ denotes the exact solution of~\eqref{eqn:q_NLSE_1} with the initial value $v_0$ at time $\xi$. The second term in~\eqref{eqn:fp_quadratic} can be computed exactly in Fourier space
\begin{equation*}
\begin{aligned}
    &-i\mu\int_{0}^{\Tau} \e^{i\left( \Tau -s \right) \partial_{x}^{2}}\big( \e^{is\partial_{x}^{2}}u\left( T_{n} \right) \big)^{2} \d s\\
    & =-2i\mu \Tau \hat{u}_{0}^{n} \e^{i\Tau \partial_{x}^{2}}u\left( T_{n} \right) +\frac{\mu}{2} \big( \e^{i\Tau \partial_{x}^{2}}\partial_{x}^{-1} u\left( T_{n} \right) \big)^{2} -\frac{\mu}{2} \e^{i\Tau \partial_{x}^{2}}\left( \partial_{x}^{-1} u\left( T_{n} \right) \right)^{2}+ i\mu \Tau (\widehat{u(T_n)}_0)^2.
    \end{aligned}
\end{equation*}
Inserting this identity into~\eqref{eqn:fp_quadratic} and discarding the high-order term $-i\mu {\rm R}_1$ yields the scheme~\eqref{eqn:LRI_q1}.

\subsection{Convergence analysis}
We now present the convergence analysis of the parareal method for the quadratic NLS \eqref{eqn:q_NLSE_1}, using the exponential low-regularity integrator \eqref{eqn:LRI_q1} as the CP and the exact solver as the FP.

\begin{thm}\label{thm:q1_conv}
Assume that problem~\eqref{eqn:q_NLSE_1} admits an exact solution
$u\in C([0,T];H^r(\mathbb{T}))$ with $r>1/2$. Let $U_n^k$ denote the parareal approximation obtained with the CP~\eqref{eqn:LRI_q1} and the exact FP at the $k$-th iteration and the $n$-th coarse grid point, and let $U_n=u(T_n)$. Then there exist constants $\Tau_0$, $C_1$, and $C_2$, independent of $n$ and $k$, such that, for any $0<\Tau\leq \Tau_0$,
\begin{equation}\label{eqn:q1_conv}
\max_{1\le n\le N_c}\|U_n^k-U_n\|_{H^r}
\leq C_1(C_2\Tau)^k .
\end{equation}
\end{thm}

\begin{proof}

Since $u \in C([0,T];H^{r} (\mathbb{T}))$, we set 
$M = \max_{t\in[0,T]} \| u(t)\|_{H^r}$ and choose $R = M+2\rho$ with some $\rho>0$. Therefore, $\max_{0\leq n \leq N_c} \| U_n\|_{H^r} \le M$ since $U_n = u(T_n)$.  

We first prove by a bootstrap argument that, for sufficiently small $\Tau$, 
\begin{equation}\label{eqn:bound-Ukn-2}
\max_{k\geq 0,~0\leq n\leq N_c} \|U_n^k\|_{H^r} \leq R.
\end{equation}
Assume that the above bound holds for the iterations $0,\dotsc,k$. For the $(k+1)$-th iteration, suppose that $U_j^{k+1}\in B_R := \{u\in H^{r}: \|u\|_{H^r}\leq R\}$ for $0\leq j \leq n$. Recall that the parareal algorithm reads
    \begin{equation*}
        U_{n+1}^{k+1} = \G_{\Tau} (U_n^{k+1}) + (\F_{\Tau}-\G_{\Tau}) (U_n^k),
    \end{equation*}
    and the exact solution satisfies    \begin{equation*}
        U_{n+1} = \G_{\Tau}(U_n)+ (\F_{\Tau}- \G_{\Tau}) (U_n).
    \end{equation*}
Then the parareal error $E_n^k = U_n ^k - U_n$ satisfies
\begin{equation}\label{eqn:para err1}
    E_{n+1}^{k+1} = \G_{\Tau} (U_n^{k+1}) - \G_{\Tau} (U_n) + [(\F_{\Tau} - \G_{\Tau})(U_n^k) - (\F_{\Tau} - \G_{\Tau})(U_n)].
\end{equation}
We denote the nonlinear part of the CP as $\Phi_1 (\Tau;v) = \G_{\Tau}(v)-\e^{i\Tau  \partial_x^2} v$ and further iterate~\eqref{eqn:para err1},
\begin{equation}\label{eqn:para err2}
\begin{aligned}
    E_{n+1}^{k+1} &=  \e^{i\Tau \partial_x^2} E_n^{k+1} + \Phi_1(\Tau;U_n^{k+1}) - \Phi_1(\Tau;U_n) + [(\F_{\Tau} - \G_{\Tau})(U_n^k) - (\F_{\Tau} - \G_{\Tau})(U_n)] \\
    & = (\e^{i\Tau \partial_x^2})^{n+1}  E_0^{k+1}+ \sum_{j=0}^n (\e^{i\Tau \partial_x^2})^j (\Phi_1(\Tau;U_{n-j}^{k+1}) - \Phi_1(\Tau;U_{n-j}))\\
    &\quad + \sum_{j=0}^n (\e^{i\Tau \partial_x^2})^j [(\F_{\Tau} - \G_{\Tau})(U_{n-j}^k) - (\F_{\Tau} - \G_{\Tau})(U_{n-j})],
    \end{aligned}
\end{equation}
where $E_0^{k+1} = U_0^{k+1}-U_0=0$. 

Now we claim the following two key estimates, for $w_0,v_0\in B_R$ and $r>1/2$,
\begin{align}
    \| \Phi_1 (\Tau;w_0) - \Phi_1 (\Tau;v_0)\|_{H^r} &\leq C_R \Tau \|w_0-v_0\|_{H^r} \label{eqn:stab.},\\
  \|(\F_{\Tau} - \G_{\Tau})(w_0) - (\F_{\Tau} - \G_{\Tau})(v_0)\|_{H^r}  &\leq C_R \Tau^2 \|w_0 - v_0\|_{H^r}, \label{eqn:order}
\end{align}
for some constant $C_R$ depending on the bound $R$.

Combining~\eqref{eqn:stab.} and~\eqref{eqn:order}, and taking $H^r$ norm of $E_{n+1}^{k+1}$ in~\eqref{eqn:para err2}, we obtain
\begin{align*}
    \|E_{n+1}^{k+1} \|_{H^r} \leq C_R \Tau  \sum_{j=0}^n \|E_{n-j}^{k+1}\|_{H^r} +  C_R \Tau^2 \sum_{j=0}^n \|  E_{n-j}^k \|_{H^r}.
\end{align*}
Then we apply Gronwall's inequality in subscript $n$ for parareal errors in the ${(k+1)}$-th iteration to derive
\begin{align}\label{eqn:iter}
\|E_{n+1}^{k+1} \|_{H^r} \leq C_R \e^{C_RT} \Tau^2 \sum_{j=0}^n \|E_j^k \|_{H^r} \leq C_R{T}\e^{C_R T}  \Tau \sup_{0\leq j\leq n} \|E_{j}^k\|_{H^r}\leq  C_R{T}\e^{C_R T}  \Tau (R+M),
\end{align}
where $n\Tau \leq T$. Taking $\Tau$ sufficiently small such that $C_R {T} \e^{C_R T}  \Tau (R+M) \leq \rho$, which gives $$\|U_{n+1}^{k+1}\|_{H^r} \leq M+\rho <R\quad \text{for all} ~~ 1 \le n \le N_c.$$ 
This proves \eqref{eqn:bound-Ukn-2} and therefore all parareal solutions remain in $B_R$. 

Iterating the  inequality~\eqref{eqn:iter} gives
\begin{equation*}
   \|E_{n+1}^{k+1} \|_{H^r} \leq  \sup_{0\leq j \leq N_c} \|E_j^0 \|_{H^r} (C_R {T} \e^{C_R T}  \Tau )^{k+1}.  
\end{equation*}
The desired result holds with $C_1 = \sup_{0\leq j \leq N_c} \|E_j^0 \|_{H^r}$ and $C_2 = C_R {T}\e^{C_R T}$.

Finally, we prove the claims \eqref{eqn:stab.} and \eqref{eqn:order}. Note that the constant $C_R$ may change from line to line.
By the bilinear estimate~\eqref{eqn:bilinear} with $d=1$, we observe 
\begin{align*}
&\| \Phi_{1} \left( \Tau ;w_{0} \right) -\Phi_{1} \left( \Tau ;v_{0} \right) \|_{H^{r}} \leq |\mu |\| \int_{0}^{\Tau} \e^{i\left( \Tau -s \right) \partial_{x}^{2}}\left( (\e^{is\partial_x^2}w_{0})^{2}-(\e^{is\partial_x^2}v_{0})^{2} \right) \d s\|_{H^{r}}\\ 
&\quad \leq C\Tau |\mu |\| w_{0}+v_{0}\|_{H^{r}} \| w_{0}-v_{0}\|_{H^{r}} \leq C_R \Tau \| w_{0}-v_{0}\|_{H^{r}}.
\end{align*}
This shows the estimate \eqref{eqn:stab.}.

Next, we turn to the proof of the estimate \eqref{eqn:order}. Denoting $w(s)$ and $v(s)$ as the exact solutions of~\eqref{eqn:q_NLSE_1} with the initial values $w_0$ and $v_0$ at time $s$, respectively,  we observe that
\begin{equation*}
\begin{split}
    &(\F_{\Tau} - \G_{\Tau})(w_0) - (\F_{\Tau} - \G_{\Tau})(v_0)\\
    &= -i\mu (\int_{0}^{\Tau} \e^{i\left( \Tau -s \right) \partial_{x}^{2}}\big( w\left( s \right)^{2} -(\e^{is\partial_{x}^{2}}w_{0})^2 \big) \d s-\int_{0}^{\Tau} \e^{i\left( \Tau -s \right) \partial_{x}^{2}}\big( v\left( s \right)^{2} -(\e^{is\partial_{x}^{2}}v_{0})^2 \big) \d s)\\
    &= -i\mu ({\rm R}_1 (\Tau;w_0) - {\rm R}_1 (\Tau;v_0) ).
\end{split}
\end{equation*}
Applying the $H^r$ norm on both sides of the above relation gives
\begin{align*}
    &\quad~ \| (\F_{\Tau} - \G_{\Tau})(w_0) - (\F_{\Tau} - \G_{\Tau})(v_0) \|_{H^r}\\
    &= |\mu| \|{\rm R}_1 (\Tau;w_0) - {\rm R}_1 (\Tau;v_0) \|_{H^r} 
     \leq \Tau |\mu|\sup_{0\leq s\leq \Tau} \| \big( w\left( s \right)^{2} -( \e^{is\partial_{x}^{2}}w_{0})^{2} \big) -\big( v\left( s \right)^{2} -( \e^{is\partial_{x}^{2}}v_{0} )^{2} \big) \|_{H^{r}}\\
    & =  \Tau  |\mu| \sup_{0\leq s \leq \Tau} \Big\| (w(s) + \e^{is\partial_x^2}w_0) (-i\mu \int_0^s \e^{i(s-\xi)\partial_x^2} w(\xi)^2 \d \xi)\\
    &\mspace{120mu} - (v(s) + \e^{is\partial_x^2}v_0) (-i\mu \int_0^s \e^{i(s-\xi)\partial_x^2} v(\xi)^2 \d \xi)\Big\|_{H^r} \\
    &=: \Tau |\mu| \sup_{0\leq s\leq \Tau} \| {\rm I}_{1,1} (s) \cdot {\rm I}_{1,2} (s) - {\rm I}_{1,3} (s) \cdot {\rm I}_{1,4} (s) \|_{H^r}.
\end{align*}
Through the triangle inequality and the bilinear estimate~\eqref{eqn:bilinear} again, the following inequalities hold for $0 \leq s\leq \Tau$,
\begin{equation}\label{eqn:I1}
    \begin{aligned}
&\| {\rm I}_{1,1} (s) \cdot {\rm I}_{1,2} (s) - {\rm I}_{1,3} (s) \cdot {\rm I}_{1,4} (s) \|_{H^r} \\
&\leq \| {\rm I}_{1,1}(s)\|_{H^r} \| {\rm I}_{1,2}(s) -  {\rm I}_{1,4} (s)\|_{H^r} + \| {\rm I}_{1,4}(s)\|_{H^r} \| {\rm I}_{1,1}(s) -  {\rm I}_{1,3} (s)\|_{H^r}. 
\end{aligned}
\end{equation}

We first prove the stability result. Apply the Bihari--LaSalle inequality~\cite{bihari1956generalization} to the following inequality
\begin{equation*}
\begin{aligned}
    \|w(s)\|_{H^r} \leq\| \e^{is\partial_x^2} w_0\|_{H^r} + |\mu|\| \int_0^{s}\e^{i(s-\xi)\partial_x^2}w(\xi)^2~\d \xi\|_{H^r}\leq \|w_0 \|_{H^r} + C|\mu| \int_0^s \|w(\xi)\|_{H^r}^2\, \d \xi
\end{aligned}
\end{equation*}
gives
\begin{equation}\label{eqn:w_stab}
    \|w(s) \|_{H^r} \leq 
\frac{\|w_0\|_{H^r}}
{1-C|\mu|\|w_0\|_{H^r}s}\leq C \|w_0\|_{H^r},\quad\text{when }0<s\leq \Tau,
\end{equation}
when $\Tau$ is sufficiently small. The boundedness of ${\rm I}_{1,1}$ comes from {the triangle inequality} and~\eqref{eqn:w_stab}
\begin{align*}
    \| {\rm I}_{1,1} (s)\|_{H^r} = \|w(s)+\e^{is\partial_x^2} w_0\|_{H^r} \leq \|w(s) \|_{H^r} + \|w_0\|_{H^r}  \le C\|w_0\|_{H^r} \le C_R.
\end{align*}
Similar to~\eqref{eqn:w_stab}, we also conclude that $\|v(s)\|_{H^r} \leq C\| v_0\|_{H^r} $ and the following estimate holds
$$\|{\rm I}_{1,4}(s)\|_{H^r}\leq  \Tau \|v_0\|_{H^r}^2 \le C_R \Tau .$$ 
Then we estimate for ${\rm I}_{1,2} - {\rm I}_{1,4}$ in~\eqref{eqn:I1}, 
\begin{align*}
    \|{\rm I}_{1,2} (s) - {\rm I}_{1,4}(s)\|_{H^r} &\leq C|\mu| \int_0^s \|w(\xi) +v(\xi)\|_{H^r} \|w(\xi)-v(\xi)\|_{H^r} \,\d \xi \\
    &\leq C_R |\mu| \int_0^{\Tau} \|w(\xi)-v(\xi)\|_{H^r} \,\d \xi. 
\end{align*}
Furthermore, by the Duhamel's principle and the stability estimate~\eqref{eqn:w_stab}:
\begin{equation}\label{eqn:w-v_stab}
\begin{aligned}
    \| w(\xi) - v(\xi)\|_{H^r}&\leq \|\e^{i\xi \partial_x^2} (w_0-v_0)\|_{H^r}+  |\mu| \int_0^\xi \| \e^{i(\xi-\eta)\partial_x^2} (w(\eta)^2 -v(\eta)^2)\|_{H^r}\, \d \eta\\
    &\leq \|w_0 -v_0\|_{H^r} + |\mu| \int_0^\xi \|w(\eta)+v(\eta)\|_{H^r} \|w(\eta) - v(\eta)\|_{H^r}\, \d \eta\\
    & \leq \|w_0 - v_0\|_{H^r} + |\mu|C(\|w_0\|_{H^r}+ \|v_0\|_{H^r}) \int_0^\xi \|w(\eta) - v(\eta)\|_{H^r} \, \d \eta\\
    &\leq \|w_0 - v_0\|_{H^r} + C_R \int_0^\xi \|w(\eta) - v(\eta)\|_{H^r} \, \d \eta.
\end{aligned}
\end{equation}
The Gronwall's inequality yields $\|w(s) - v(s) \|_{H^r} \leq C_R \|w_0 -v_0 \|_{H^r}$ and hence 
\begin{align*}
\max_{s\in [0,\Tau]}\|{\rm I}_{1,2} (s) - {\rm I}_{1,4}(s)\|_{H^r} &\leq C_R \Tau  \|w_0 -v_0 \|_{H^r}.
\end{align*}
Therefore 
$$ \| {\rm I}_{1,1}(s)\|_{H^r} \| {\rm I}_{1,2}(s) -  {\rm I}_{1,4} (s)\|_{H^r} \le C_R \Tau \| w_0 - v_0 \|_{H^r}. $$
Similarly, we also conclude that 
$$ \max_{s\in [0,\Tau]}\| {\rm I}_{1,1}(s) - {\rm I}_{1,3}(s) \|_{H^r} \leq C_R\|w_0 - v_0 \|_{H^r}.$$  
Hence 
$$ \| {\rm I}_{1,4}(s)\|_{H^r} \| {\rm I}_{1,1}(s) -  {\rm I}_{1,3} (s)\|_{H^r} \le C_R \Tau \| w_0 - v_0 \|_{H^r}. $$
Consequently, combining these estimates with \eqref{eqn:I1} shows that the inequality \eqref{eqn:order} holds.
Then we complete the proof of the theorem.

\end{proof}

\begin{remark}
The constant $C_2$ is independent of $n$, $k$, and $\Tau$, and depends only on the a priori bound $R$, the final time $T$, the regularity index $r$, and the parameter $\mu$. On the other hand, $C_1$ depends on the initial parareal error
$\sup_{0\leq j\leq N_c}\|E_j^0\|_{H^r}$,
and hence on the initialization of the algorithm. In particular, whether the parareal iteration is initialized by the CP or simply by the initial condition, the convergence with respect to the iteration number $k$ is always governed by the factor $(C_2\Tau)^k$. This convergence behavior is fully supported by the numerical experiments; see, for example, Figure~\ref{fig:Q_NLSE_1}.
\end{remark}

{
\begin{remark}\label{rem:general}
We note that, in the proof of Theorem \ref{thm:q1_conv}, the key steps are to establish the stability estimate \eqref{eqn:stab.} and the consistency estimate \eqref{eqn:order}. The stability estimate \eqref{eqn:stab.} can typically be obtained for standard exponential integrators. By contrast, the consistency estimate \eqref{eqn:order} is more restrictive, as it involves Sobolev norms of the same order on both sides. This necessitates the use of a specially designed low-regularity integrator, which achieves the optimal convergence rate without any loss of spatial regularity \cite[Eq.~(35)]{MR3807360}. This observation motivates the development of the exponential low-regularity parareal algorithm and the general convergence framework presented in Section~\ref{sec:general case}.
\end{remark}}

\subsection{Numerical illustration} 
In this subsection, we numerically validate the convergence result established in Theorem~\ref{thm:q1_conv}. 
We first consider problem~\eqref{eqn:q_NLSE_1} with
\[
u_0(x)=\mathrm{e}^{-x^2/2}+1\in H^{\frac12-}(\mathbb{T}), 
\qquad \mu=-1, 
\qquad T=1.
\]
The same ELRI scheme~\eqref{eqn:LRI_q1} is used as the FP with time step $\tau=10^{-4}$. We compare the performance of three CPs: the ELRI scheme~\eqref{eqn:LRI_q1}, Strang splitting~\cite[Eq.~(44)]{MR3807360}, and the third-order exponential Runge--Kutta method (ERK3)~\cite{MR1894772}; see Figure~\ref{fig:Q_NLSE_order}. 
The parareal iteration is initialized by the initial datum, namely $U_n^0=u_0$ for all $n$.
As predicted by Theorem~\ref{thm:q1_conv}, when the ELRI scheme~\eqref{eqn:LRI_q1} is used as the CP, the $L^2$ parareal error at the $k$-th iteration exhibits an order of $\mathcal{O}(\Tau^k)$. In contrast, for the other two CPs, the dependence of the parareal error on $\Tau$ appears to be largely independent of the iteration index $k$, thereby preventing rapid parareal convergence. These results demonstrate the superior convergence behavior of the exponential low-regularity parareal algorithm in case of low-regularity data.

\begin{figure}
    \centering
\includegraphics[width=0.98\textwidth,trim={0.8cm 0.8cm 0.5cm 0.75cm},clip]{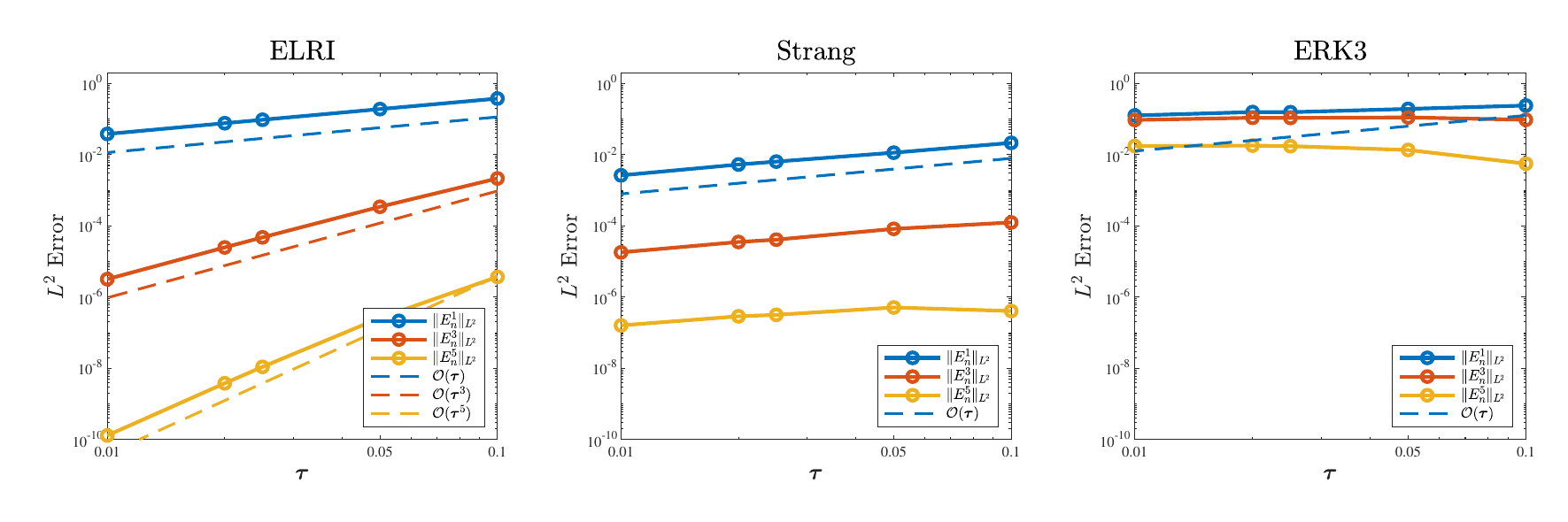}
    \caption{$L^2$ error versus coarse time step $\Tau \in \{0.1, 0.05,0.025, 0.02, 0.01\}$ for the parareal applied to quadratic nonlinear Schr\"{o}dinger equations of type \eqref{eqn:q_NLSE_1} ($H^{\frac12}$ initial data, $\mu = -1$) with three CPs: ELRI~\eqref{eqn:LRI_q1} (left), Strang splitting (middle), and ERK3 (right).}
    \label{fig:Q_NLSE_order}
\end{figure}

Next, we test the convergence in case of more regular initial data and the same FP configuration, where we let 
$$u_0(x) = |x-\pi|^{1/2} \sin (x-\pi) \in H^{2-}(\mathbb{T}),\qquad \mu=1,\qquad T=5,$$
in the quadratic NLS \eqref{eqn:q_NLSE_1}. 
{In this case, the solution has additional regularity, so classical time-stepping schemes may attain at least first-order convergence and thus become potentially comparable to the ELRI scheme~\eqref{eqn:LRI_q1}. Nevertheless, the numerical results shown in Figure~\ref{fig:Q_NLSE_1} demonstrate that the parareal algorithm with the ELRI as the CP significantly outperforms those using other CPs, including Lie splitting, Strang splitting, ERK1, and ERK3. This is consistent with the analysis in the proof of Theorem~\ref{thm:q1_conv}, where the consistency estimate \eqref{eqn:order} is established only for the ELRI~\eqref{eqn:LRI_q1}. This example highlights the necessity of using an ELRI as the CP, even for higher-regularity initial data. In particular, the numerical results show that the exponential low-regularity parareal algorithm converges linearly, with a contraction factor that decreases as $\Tau$ decreases, thereby fully supporting Theorem~\ref{thm:q1_conv}. }

\begin{figure}
    \centering
\includegraphics[width=0.98\textwidth,trim={0.8cm 0.8cm 0.5cm 0.75cm},clip]{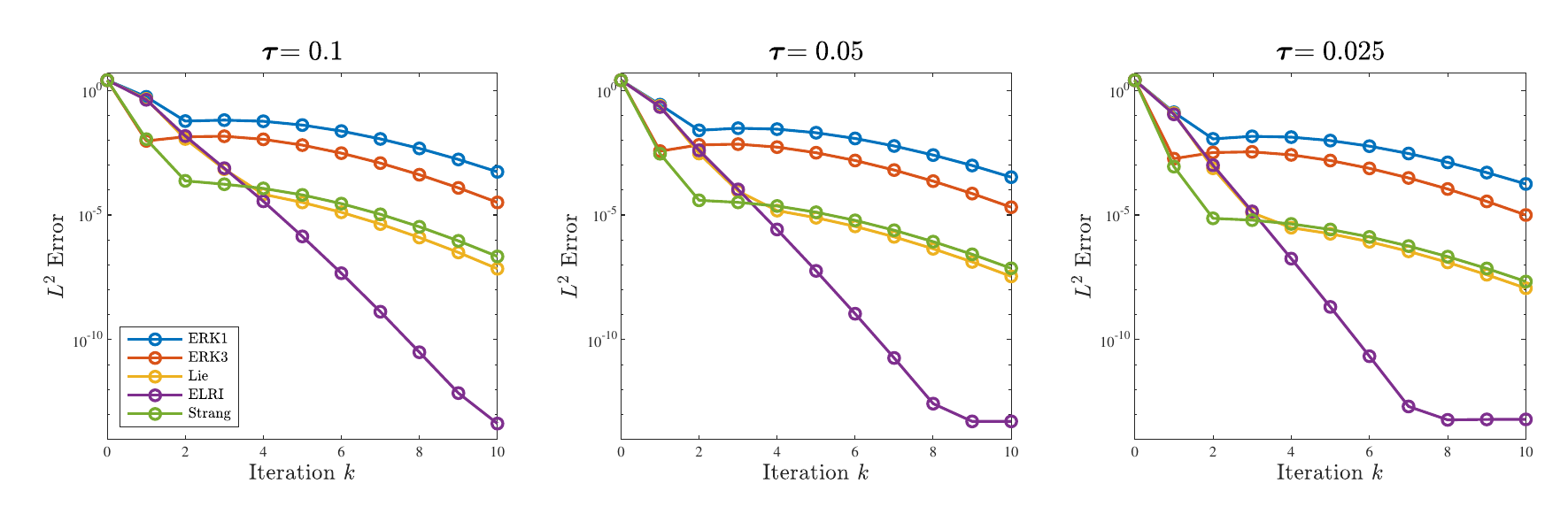}
    \caption{ $L^2$ error versus iteration $k$ for the parareal applied to quadratic nonlinear Schr\"{o}dinger equations of type~\eqref{eqn:q_NLSE_1} ($H^{2-}$ initial data, $\mu = 1$) with five CPs, and $\Tau \in \{0.1,0.05,0.025\}.$}
    \label{fig:Q_NLSE_1}
\end{figure}

\section{Exponential low-regularity parareal algorithms}\label{sec:general case}
Now, we are ready to consider the general form of the NLS equation \eqref{eqn:NLSE} and establish a general convergence framework for the parareal algorithms. In particular, we assume that the CP is of the following exponential-type:
\begin{equation}\label{eqn:cp_}
    \G_{\Tau} (u) = \e^{i \Tau \Delta} u + \Phi_c (\Tau;u),
\end{equation}
where $\Phi_c$ propagates the nonhomogeneous parts. Recall that the FP is assumed to be the exact solver and is defined in~\eqref{eqn:fp_}. Motivated by the proof in Theorem~\ref{thm:q1_conv}, we make the following assumptions on the CPs~\eqref{eqn:cp_}: 
\begin{assum}\label{assum:cp}
Let $r > d/2$.
The following inequalities hold for $w_0, v_0 \in B_R$ and $0< \Tau \leq 1$, where $B_R = \{ u \in H^r: \|u\|_{H^r}\leq R\}$:
    \begin{enumerate}
        \item[(i)] $\| \Phi_c (\Tau;w_0) - \Phi_c (\Tau;v_0) \|_{H^r} \leq C_R \Tau \|w_0-v_0\|_{H^r}$,
        \item[(ii)] $\| (\F_{\Tau} - \G_{\Tau})(w_0) - (\F_{\Tau} - \G_{\Tau})(v_0)\|_{H^r} \leq C_R \Tau^2 \|w_0 - v_0\|_{H^r}$,
    \end{enumerate}
where the constant $C_R $ is independent of $w_0,v_0,$ and $\Tau$. 
\end{assum}
The next theorem demonstrates the convergence of the parareal algorithm. The proof is similar to that of Theorem~\ref{thm:q1_conv} and we omit it here. 
\begin{thm}\label{thm:general_conv}
Let Assumption~\ref{assum:cp} hold, and assume that problem~\eqref{eqn:NLSE} admits an exact solution
$u\in C([0,T];H^r(\mathbb{T}^d))$ with $r>d/2$. Let $U_n^k$ denote the parareal approximation obtained with the CP~\eqref{eqn:cp_} and the FP~\eqref{eqn:fp_} at the $k$-th iteration and the $n$-th coarse grid point, and let $U_n=u(T_n)$. Then there exist constants $\Tau_0$, $C_1$, and $C_2$, independent of $n$ and $k$, such that, for any $0<\Tau\leq \Tau_0$, 
\begin{equation}\label{eqn:c1 c2}
   \max_{1\leq n\leq N_c} \|U_n^k-U_n\|_{H^r}
    \leq C_1(C_2\Tau)^k .
\end{equation}
\end{thm}

The key ingredient is the second condition in Assumption~\ref{assum:cp}, which serves as a sufficient condition to ensure the linear convergence stated in Theorem~\ref{thm:general_conv}. This assumption implies that the local truncation error of the CP~\eqref{eqn:cp_} achieves $\mathcal{O}(\Tau^2)$ \textit{without loss of spatial regularity}. In the following section, we verify Assumption \ref{assum:cp} through examples involving one-dimensional quadratic Schr\"{o}dinger equation with $F(u,\bar u) = \mu  u\bar{u} = \mu |u|^2$ and the cubic Schr\"{o}dinger equation with $F(u,\bar u) = \mu u^2 \bar u = \mu |u|^2 u$ on $\mathbb{T}$. 

\subsection{Quadratic NLS with $F(u,\bar u) = \mu  u\bar{u}$}
In this subsection, we consider the following  quadratic Schr\"{o}dinger equation~\cite{MR1357398},
\begin{equation}\label{eqn:q_NLSE_2}
    i\partial_t u = -\partial_x^2 u + \mu |u|^2,\quad (t,x) \in  (0,T) \times \mathbb{T},
\end{equation}
with the initial value $u(0,x) = u_0 (x) \in H^{\frac12+} (\mathbb{T})$. This NLS equation corresponds to $F(u,\overline{u}) = \mu u\bar{u}$ and $d=1$ in~\eqref{eqn:NLSE}. The CP is of the exponential type~\eqref{eqn:cp_} and the nonhomogeneous term is given by \cite[Eq. (37)]{MR3807360}:
\begin{equation}\label{eqn:LRI_q2}
\begin{aligned}
    \Phi_{c,1}(\Tau;u)&=\left( -i\mu \Tau \overline{\hat{u}_{0}} \right) \e^{i\Tau \partial_{x}^{2}}u-i\mu \Tau (2\pi)^{-1} \| u\|_{L^{2}}^{2}  + i\mu \Tau |\hat{u}_0|^2\\
    &\quad +\frac{\mu}{2} \partial_{x}^{-1} \left[ \left( \e^{i\Tau \partial_{x}^{2}}u \right) \left( \e^{-i\Tau \partial_{x}^{2}}\partial_{x}^{-1} \overline{u} \right) - \e^{i\Tau \partial_{x}^{2}}\left( u\partial_{x}^{-1} \overline{u} \right) \right].
    \end{aligned}
\end{equation}

Note that we add the missing overlap frequency term ($k_1 = k_2 =0$ in $I^{\tau} (w,t_n)$ in \cite[Eq. (36)-(37)]{MR3807360}) $i\mu \Tau |\hat{u}_0|^2$ and rescale the $L^2$ norm term to match the domain. Through the construction of $\Phi_{c,1}$, we have a clearer expression: 
\begin{equation}
    \Phi_{c,1} \left( \Tau ;u \right) =-i\mu \int_{0}^{\Tau} \e^{i\left( \Tau -s \right) \partial_{x}^{2}}|\e ^{is\partial_x^2}u|^{2}\, \d s.
\end{equation}

Next, we aim to verify Assumption~\ref{assum:cp} for the CP~\eqref{eqn:LRI_q2}. 
\begin{lem}\label{lem:quadratic}
With the CP~\eqref{eqn:LRI_q2} and the FP~\eqref{eqn:fp_} applied to problem~\eqref{eqn:q_NLSE_2}, Assumption~\ref{assum:cp} holds. 
\end{lem}

\begin{proof}
The condition (i) in Assumption \ref{assum:cp} is trivial, since 
    \begin{equation*}
    \begin{aligned}
    &\| \Phi_{c,1} \left( \Tau ;w_0 \right) -\Phi_{c,1} \left( \Tau ;v_0 \right) \|_{H^{r}} \leq |\mu |\| \int_{0}^{\Tau} \e^{i\left( \Tau - s \right) \partial_{x}^{2}}\left( |\e^{is\partial_x^2}w_{0}|^{2}-|\e^{is\partial_x^2} v_{0}|^{2} \right)\, \d s\|_{H^r}\\
    & \leq |\mu| \Tau (\|w_0\|_{H^r} \|\bar{w}_0-\bar{v}_0\|_{H^r}+\| \bar{v}_0\|_{H^r} \|w_0 -v_0\|_{H^r}) \leq C_R \Tau \| w_{0}-v_{0}\|_{H^r} .
    \end{aligned}
    \end{equation*}
Then we turn to the condition (ii) and observe the split:
\begin{align*}
&\left( \F_{\Tau}- \G_{\Tau} \right) \left( w_{0} \right) -\left( \F_{\Tau}-\G_{\Tau} \right) \left( v_{0} \right)\\
&=-i\mu \int_{0}^{\Tau} \e^{i\left( \Tau -s \right) \partial_{x}^{2}}\left[ \left( |w\left( s \right) |^{2}-|\e^{is\partial_{x}^{2}}w_{0}|^{2} \right) -\left( |v\left( s \right) |^{2}-|\e^{is\partial_{x}^{2}}v_{0}|^{2} \right) \right] \d s\\
&=-i\mu \int_0^{\Tau} \e^{i\left( \Tau -s \right) \partial_{x}^{2}} \Big[ w(s) (\bar{w}(s) - \e^{-is\partial_x^2} \bar{w}_0) - v(s) (\bar{v}(s) - \e^{-is\partial_x^2} \bar{v}_0) \\
&\quad + \e^{-is\partial_x^2} \bar{w}_0 (w(s) - \e^{is\partial_x^2}w_0) -  \e^{-is\partial_x^2} \bar{v}_0 (v(s) - \e^{is\partial_x^2}v_0)\Big] \, \d s\\
& =: -i\mu  \int_0^{\Tau} \e^{i\left( \Tau -s \right) \partial_{x}^{2}} \Big[ w(s)\cdot \overline{{\rm I}_{2,1}}(s) - v(s)\cdot \overline{{\rm I}_{2,2}}(s) \\
&\quad + \e^{-is\partial_x^2}\bar{w}_0 \cdot {{\rm I}}_{2,1}(s) - \e^{-is\partial_x^2}\bar{v}_0\cdot{{\rm I}}_{2,2}(s)\Big] \, \d s.
\end{align*}
With $r>1/2$, applying the $H^r$ norm on both sides of the above relation and denoting $w(s)$ and $v(s)$ as the exact solutions of~\eqref{eqn:q_NLSE_2} with the initial values $w_0$ and $v_0$ at time $s$
yields 
\begin{equation}\label{eqn:F-G}
\begin{aligned}
   &\| \left( \F_{\Tau}- \G_{\Tau} \right) \left( w_{0} \right) -\left( \F_{\Tau}-\G_{\Tau} \right) \left( v_{0} \right) \|_{H^r}\\
   & \leq |\mu| \Tau \sup_{0\leq s \leq \Tau} \Big( \|w(s)\|_{H^r} \| \overline{{\rm I}_{2,1}}(s) - \overline{{\rm I}_{2,2}}(s)\|_{H^r} + \|\overline{{\rm I}_{2,2}}(s)\|_{H^r} \| w(s) - v(s)\|_{H^r} \\
   &\quad + \|\overline{w_0}\|_{H^r} \| {{\rm I}}_{2,1}(s) - {{\rm I}}_{2,2}(s)\|_{H^r} + \|{{\rm I}}_{2,2}(s)\|_{H^r} \| \e^{-is\partial_x^2} (\overline{w_0} - \overline{v_0})\|_{H^r}\Big).
\end{aligned}
\end{equation}
The boundedness of $w(s)$ and $v(s)$ follows from the same proof in~\eqref{eqn:w_stab}. The stability estimate $\|w(s) -v(s)\|_{H^r} \leq C_R \|w_0 -v_0\|_{H^r}$ follows from the same proof in~\eqref{eqn:w_stab}. By~\eqref{eqn:bilinear}, $\overline{\rm I_{2,2}}(s)$ can be bounded:
\begin{equation*}
    \| \overline{\rm I_{2,2}}(s)\|_{H^r}\leq |\mu| \|\int_0^s \e^{i(s-\xi)\partial_x^2} |v(\xi)|^2\,\d \xi\|_{H^r}\leq |\mu| \int_0^s\|v(\xi)\|_{H^r}^2\,\d \xi \leq |\mu|\Tau \sup_{0\leq s\leq \Tau} \| v(s)\|_{H^r}^2.
\end{equation*}
By collecting the estimate above, we obtain
\begin{equation}\label{eqn:F-G_1}
\begin{aligned}
   &\| \left( \F_{\Tau}- \G_{\Tau} \right) \left( w_{0} \right) -\left( \F_{\Tau}-\G_{\Tau} \right) \left( v_{0} \right) \|_{H^r}\\
    &\leq |\mu| \Tau \Big( \sup_{0\leq s \leq \Tau} \big( C_R \| \overline{{\rm I}_{2,1}}(s) - \overline{{\rm I}_{2,2}}(s)\|_{H^r} \big) +  \sup_{0\leq s \leq \Tau} C \Tau \|v(s)\|^2_{H^r} \|w_0 - v_0 \|_{H^r}\\
   &\quad +  \sup_{0\leq s \leq \Tau} C_R\| {{\rm I}}_{2,1}(s) - {{\rm I}}_{2,2}(s)\|_{H^r} +   \sup_{0\leq s \leq \Tau} C\Tau \|v(s)\|_{H^r}^2 \| \e^{-is\partial_x^2} ({w}_0 - {v}_0)\|_{H^r}\Big)\\
   &\leq C_R\Tau \Big( \sup_{0\leq s\leq \Tau} (\| {{\rm I}}_{2,1}(s) - {{\rm I}}_{2,2}(s)\|_{H^r} + \Tau  \| w_0 - v_0 \|_{H^r}) \Big),
\end{aligned}
\end{equation}
It remains to estimate the term $\sup_{0\leq s\leq \Tau} \| {{\rm I}}_{2,1}(s) - {{\rm I}}_{2,2}(s)\|_{H^r}$,
\begin{align*}
    &  \| {{\rm I}}_{2,1}(s) - {{\rm I}}_{2,2}(s)\|_{H^r}\leq |\mu| \| \int_0^{s} \e^{-i(s-\xi)\partial_x^2} \big( |w(\xi)|^2 - |v(\xi)|^2 \big)\, \d \xi\|_{H^r} \\
    &\quad \leq C_R \Tau \sup_{0\leq s\leq \Tau}\|w(s) - v(s) \|_{H^r} \leq C_R \Tau \|w_0 - v_0 \|_{H^r}.
\end{align*}
This combines with~\eqref{eqn:F-G_1} to show the second condition.  
\end{proof}

As a direct result of Lemma \ref{lem:quadratic} and Theorem \ref{assum:cp}, we have the following corollary for the convergence of the proposed exponential low-regularity parareal algorithm.

\begin{coro}\label{cor:41_conv}
Assume that the problem~\eqref{eqn:q_NLSE_2} admits an exact solution
$u\in C([0,T];H^r(\mathbb{T}))$ with $r>1/2$. Let $U_n^k$ denote the parareal approximation obtained by using the CP~\eqref{eqn:cp_}, augmented by the nonhomogeneous term~\eqref{eqn:LRI_q2}, together with the exact FP~\eqref{eqn:fp_}, at the $k$-th iteration and the $n$-th coarse grid point, and let $U_n=u(T_n)$ be the exact solution to the problem~\eqref{eqn:q_NLSE_2}. Then there exist constants $\Tau_0$, $C_1$, and $C_2$, independent of $n$ and $k$, such that, for any $0<\Tau\leq \Tau_0$,
\begin{equation}\label{eqn:q2_conv}
\max_{1\le n\le N_c} \|U_n^k-U_n\|_{H^r}
\leq C_1(C_2\Tau)^k .
\end{equation}
\end{coro}

\subsection{Cubic NLS with $F(u,\bar u) = \mu |u|^2 u$}
In this subsection, we consider the following cubic NLS on the torus:
\begin{equation}\label{eqn:cubic_1d}
    i\partial_t u = -\partial_x^2 u + \mu |u|^2 u,\quad (t,x) \in (0,T) \times \mathbb{T},
\end{equation}
with initial data $u(0,x)=u_0(x) \in H^{\frac32+} (\mathbb{T})$. This NLS corresponds to $F(u,\overline{u}) = \mu u^2 \bar u=\mu |u|^2 u$ and $d=1$ in~\eqref{eqn:NLSE}.

In the parareal algorithm \eqref{eqn:parareal}, we employ the exponential low-regularity integrator constructed in \cite[(Eq.~(1.11))]{MR4405493} as the CP and use the exact solver as the FP. This CP is of the exponential form~\eqref{eqn:cp_}, and  the nonhomogeneous term is given by
\begin{equation}\label{eqn:cp_2}
    \begin{aligned}
\Phi_{c,2} ( \Tau;u ) &=\e^{i\Tau \partial_{x}^{2}}( \e^{-2i\mu\Tau (M(u)+P(u)\partial_{x}^{-1}) }-1)u-i\mu \Tau \Pi_{0} \left( |u|^{2}u \right) +2i\mu \Tau M (u)\Pi_{0} \left( u \right)\\
&\quad - \frac{\mu}{2} \partial_{x}^{-2} \left[ \left( \e^{-i\Tau \partial_{x}^{2}}\bar{u} \right) \cdot \e^{i\Tau \partial_{x}^{2}} (u^{2}) \right] +\frac{\mu}{2}  \e^{i\Tau \partial_{x}^{2}}\partial_{x}^{-2} \left[ |u|^{2}u \right]\\
&\quad +\mu \partial_{x}^{-1} \left[ \left( \e^{i\Tau \partial_{x}^{2}}u \right) \cdot \partial_{x}^{-1} \left( |\e^{i\Tau \partial_{x}^{2}}u|^{2} \right) \right] - \mu \e^{i\Tau \partial_{x}^{2}}\partial_{x}^{-1} \left[ u\cdot \partial_{x}^{-1} \left( |u|^{2} \right) \right].
\end{aligned}
\end{equation}
Here $M(u)$ and $P(u)$ denote the mass and momentum, respectively, which are conserved quantities of the continuum NLS model. They are defined by
\begin{equation*}
    M (u) = \Pi_0 (|u|^2),~P(u) = \Pi_0 (u \partial_x \overline{u}),\quad \text{where}~\Pi_0(f) = \frac{1}{2\pi} \int_{\mathbb{T}}f(x)\, \d x.
\end{equation*}
Here we assume that $M(u_0)>0$. For the scheme~\eqref{eqn:cp_2}, if $M(u)$ and $P(u)$ are replaced by their initial values $M(u_0)$ and $P(u_0)$, respectively, then it was proved in~\cite{MR4405493} that the corresponding scheme converges to the exact solution with first-order accuracy in
$H^{\frac{3}{2}+}(\mathbb{T})$, without any loss of spatial regularity. In the present work, we instead allow the quantities (M) and (P) to depend on the current input, which is the straightforward extension to the parareal setting. See Remark~\ref{rmk:CP_c2} for further details.

Note that to further preserve the mass while requiring as little regularity as possible, we can modify the scheme \eqref{eqn:cp_2} by replacing $\Phi_{c,2} (\Tau;u)$ with
\begin{equation}\label{eqn:cp_3}
    \Phi_{c,3} (\Tau;u) = \Phi_{c,2} (\Tau;u) + G_1 (\Tau;u) + G_2 (\Tau;u),
\end{equation}
where the corrections $G_1$ and $G_2$ are defined by
\begin{align*}
    G_1 (\Tau;u) &= H(\Tau;u) \e^{i\Tau \partial_x^2} u,\\
    G_2 (\Tau;u) &= -\frac{1}{2} (H(\Tau;u))^2 \e^{i\Tau \partial_x^2} u-M(u_0)^{-1} H(\Tau;u)\operatorname{Re}\Pi_{0} \big(\Phi_{c,2}(\Tau;u) \e^{-i\Tau \partial_x^2} \bar{u} \big) \e^{i\Tau \partial_x^2} u, \\
    H(\Tau;u) &= -M(u_0)^{-1} \big[\operatorname{Re}\Pi_{0} \big(\Phi_{c,2}(\Tau;u)\e^{-i\Tau \partial_x^2 }\bar{u} \big) + \frac{1}{2} \Pi_0 \big( |\Phi_{c,2} (\Tau;u)|^2\big)\big].
\end{align*}
It was shown in~\cite{MR4405493} that, after this modification, the resulting ELRI preserves the mass up to fifth order when the invariants in~\eqref{eqn:cp_2} are frozen at their initial values. The analysis in this section is carried out for the CP based on \eqref{eqn:cp_2}, but it can be extended straightforwardly to the CP based on \eqref{eqn:cp_3}.


Next, we verify Assumption~\ref{assum:cp} for the CP~\eqref{eqn:cp_2}. This assumption on the modified one~\eqref{eqn:cp_3} can be verified similarly. Motivated by the argument in \cite{MR4405493}, we introduce the following notation. Let $M=M(k,k_1,k_2,\dotsc,k_m)$ be a nonnegative frequency multiplier. We denote by $\mathcal{D}_m (M;w,v)$ the class of functions $D_m(M;w,v)$ whose $k$-th Fourier coefficient satisfies 
\begin{equation}\label{eqn:D_m}
\begin{aligned}
    \mathscr{F}(D_m (M;w,v))(k)
    = \sum_{k=k_1 + \dotsc + k_m} a(k,k_1,\dotsc,k_m) \big(\hat{w}_{k_1}(t)\dotsc \hat{w}_{k_m}(t)-\hat{v}_{k_1}(t)\dotsc \hat{v}_{k_m}(t)\big),
\end{aligned}
\end{equation}
where
\begin{equation*}
    \sup_{t\in (0,T]} |a(k,k_1,\dotsc,k_m)(t)| \leq C |M(k,k_1,\dotsc,k_m)|.
\end{equation*}

In the estimates below, factors involving $u$ and $\bar u$ are treated in the
same way, since taking complex conjugates does not change the Sobolev norms or
the Fourier majorant estimates. We first present the following lemma.  

\begin{lem}\label{lem:Dm}
For any function $D_m (M;w,v) \in \mathcal{D}_m(M;w,v)$, as defined in \eqref{eqn:D_m}, the following estimates hold for $B_{R,r} = \{ w\in H^r:\|w\|_{H^r} \leq R \}$:
\begin{enumerate}
    \item[(i)] Let $\gamma \geq 1,$ and $w,v\in B_{R,\gamma}$, then 
    \begin{equation}\label{eqn:lem1}
        |\mathscr{F}D_3 (k_2 k_3;w,v)(0)| \leq C_{R,\gamma}\|w-v\|_{H^\gamma}.
    \end{equation}
    \item[(ii)] Let $\gamma > 1/2,~m\geq1$ and $w,v\in B_{R,\gamma}$, then 
    \begin{equation}\label{eqn:lem2}
        \|D_m (1;w,v)\|_{H^\gamma} \leq C_{R,\gamma,m}\|w-v\|_{H^\gamma}.
    \end{equation}
    \item[(iii)] Let $\gamma > 3/2,$ and $w,v\in B_{R,\gamma}$, then 
    \begin{equation}\label{eqn:lem3}
        \|D_3 (k^{-1}k_1 k_2 k_3;w,v)\|_{H^\gamma} \leq C_{R,\gamma}\|w-v\|_{H^\gamma}.
    \end{equation}
\end{enumerate}
\end{lem}

\begin{proof}
Let $d=w-v$. Throughout the proof, by the Fourier majorant argument and the proof in~\cite[Lemma 2.1]{MR4405493}, it suffices to estimate the corresponding expression with $\widehat f(k)$ replaced by $|\widehat f(k)|$ for each function $f$.

\noindent (i) From the definition of $D_3(M;w,v)$, we obtain
\begin{align*}
 |\mathscr{F} D_3 (k_2k_3;w,v)(0)| \leq C \sum_{k_1+k_2+k_3=0} |k_2||k_3|
 |\hat{w}_{k_1} \hat{w}_{k_2} \hat{w}_{k_3}
 - \hat{v}_{k_1} \hat{v}_{k_2} \hat{v}_{k_3} | .
\end{align*}
Using the decomposition $\hat{w}_{k_1}\hat{w}_{k_2}\hat{w}_{k_3}
-\hat{v}_{k_1}\hat{v}_{k_2}\hat{v}_{k_3}
=
\hat d_{k_1}\hat w_{k_2}\hat w_{k_3}
+\hat v_{k_1}\hat d_{k_2}\hat w_{k_3}
+\hat v_{k_1}\hat v_{k_2}\hat d_{k_3},$
we obtain
\begin{align*}
& |\mathscr{F} D_3 (k_2k_3;w,v)(0)| \\
&\leq
C \sum_{k_1+k_2+k_3=0} |k_2||k_3|
\Big(
|\hat d_{k_1}|\,|\hat w_{k_2}|\,|\hat w_{k_3}|
+|\hat v_{k_1}|\,|\hat d_{k_2}|\,|\hat w_{k_3}|+ |\hat v_{k_1}|\,|\hat v_{k_2}|\,|\hat d_{k_3}|
\Big).
\end{align*}
By Plancherel's theorem, H\"older's inequality, and the embedding
$H^\gamma(\mathbb T)\hookrightarrow L^\infty(\mathbb T)$ for
$\gamma >1/2$, the right-hand side is bounded by
\begin{align*}
& C\|d\|_{L^\infty}\||\nabla|w\|_{L^2}^2
+ C\|v\|_{L^\infty}
\||\nabla|d\|_{L^2}\||\nabla|w\|_{L^2} \\
&\quad
+ C\|v\|_{L^\infty}
\||\nabla|v\|_{L^2}\||\nabla|d\|_{L^2}.
\end{align*}
Therefore, since \(w,v\in B_{R,\gamma}\) and {$\gamma\geq 1$},
\begin{align*}
|\mathscr{F} D_3 (k_2k_3;w,v)(0)|
&\leq C_{\gamma}
\|d\|_{H^\gamma}\|w\|_{H^\gamma}^2
+ C_{\gamma}\|v\|_{H^\gamma}\|d\|_{H^\gamma}{\|w\|_{H^\gamma} }+C_{\gamma}\|v\|_{H^\gamma}^2\|d\|_{H^\gamma} \\
&\leq C_{R,\gamma}\|d\|_{H^\gamma}.
\end{align*}

\noindent (ii) We use the telescopic decomposition
\begin{equation*}
     \prod_{j=1}^m \hat w_{k_j}
    -
    \prod_{j=1}^m \hat v_{k_j}
    =
    \sum_{\ell=1}^m
    \Big(\prod_{j<\ell}\hat v_{k_j}\Big)
    \hat d_{k_\ell}
    \Big(\prod_{j>\ell}\hat w_{k_j}\Big).
\end{equation*}
Hence, by the definition of
\(D_m(1;w,v)\),
\begin{align*}
|\mathscr{F}D_m(1;w,v)(k)|
&\leq
C\sum_{\ell=1}^m
\sum_{k=k_1+\dotsc+k_m}
\Big(\prod_{j<\ell}|\hat v_{k_j}|\Big)
|\hat d_{k_\ell}|
\Big(\prod_{j>\ell}|\hat w_{k_j}|\Big).
\end{align*}
By Plancherel's identity and the standard Fourier majorant argument, we obtain 
\begin{equation*}
    \|D_m(1;w,v)\|_{H^\gamma}
\leq
C\sum_{\ell=1}^m
\|v^{\ell-1}d w^{m-\ell}\|_{H^\gamma}.
\end{equation*}
Since \(\gamma>1/2\), \(H^\gamma(\mathbb T)\) is an algebra. Thus
\begin{align*}
\|D_m(1;w,v)\|_{H^\gamma}
&\leq
C_{\gamma}\sum_{\ell=1}^m
\|v\|_{H^\gamma}^{\ell-1}
\|d\|_{H^\gamma}
\|w\|_{H^\gamma}^{m-\ell} \leq {C_{R,\gamma,m}}\|d\|_{H^\gamma}.
\end{align*}

\noindent (iii) By the definition of \(D_3\), for \(k\neq 0\),
\begin{align*}
&|\mathscr{F}D_3(k^{-1}k_1k_2k_3;w,v)(k)| \\
&\leq
C\sum_{k=k_1+k_2+k_3}
|k|^{-1}|k_1||k_2||k_3|
|\hat{w}_{k_1}\hat{w}_{k_2}\hat{w}_{k_3}
-\hat{v}_{k_1}\hat{v}_{k_2}\hat{v}_{k_3}|.
\end{align*}
Using the same telescopic decomposition as in part \((i)\), we have
\begin{align*}
&|\mathscr{F}D_3(k^{-1}k_1k_2k_3;w,v)(k)| \\
&\leq
C\sum_{k=k_1+k_2+k_3}
|k|^{-1}|k_1||k_2||k_3|
\Big(
|\hat d_{k_1}|\,|\hat w_{k_2}|\,|\hat w_{k_3}|
+|\hat v_{k_1}|\,|\hat d_{k_2}|\,|\hat w_{k_3}| \\
&\hspace{8em}
+|\hat v_{k_1}|\,|\hat v_{k_2}|\,|\hat d_{k_3}|
\Big).
\end{align*}
Since $k\neq 0$, $\langle k\rangle^\gamma |k|^{-1}
    \leq C \langle k\rangle^{\gamma-1}.$ Thus, by Plancherel's identity and the Fourier majorant argument: 
\begin{align*}
&\|D_3(k^{-1}k_1k_2k_3;w,v)\|_{H^\gamma} \\
&\leq C
\big\|
(|\nabla|d)(|\nabla|w)^2
\big\|_{H^{\gamma-1}}
+
C
\big\|
(|\nabla|v)(|\nabla|d)(|\nabla|w)
\big\|_{H^{\gamma-1}} +
C
\big\|
(|\nabla|v)^2(|\nabla|d)
\big\|_{H^{\gamma-1}} .
\end{align*}
Since $\gamma>3/2$,
$H^{\gamma-1}(\mathbb T)$ is an algebra. Therefore
\begin{align*}
&\|D_3(k^{-1}k_1k_2k_3;w,v)\|_{H^\gamma} \\
&\leq C_{\gamma}
\||\nabla|d\|_{H^{\gamma-1}}
\||\nabla|w\|_{H^{\gamma-1}}^2+ C_{\gamma}
\||\nabla|v\|_{H^{\gamma-1}}
\||\nabla|d\|_{H^{\gamma-1}}
\||\nabla|w\|_{H^{\gamma-1}} \\
&\quad
+ C_{\gamma}
\||\nabla|v\|_{H^{\gamma-1}}^2
\||\nabla|d\|_{H^{\gamma-1}} \leq {C_{R,\gamma}}\|d\|_{H^\gamma}.
\end{align*}
Hence, the proof is complete.
\end{proof}
We need the following lemma to handle the dependence of $M$ and $P$ on their input. 
\begin{lem}\label{lem:MP_stability}
Let $\gamma>3/2$ and $w,v\in B_{R,\gamma}$. Then
\begin{equation}\label{eqn:MP_stab}
    |M(w)-M(v)|+|P(w)-P(v)|
    \le C_R\|w-v\|_{H^\gamma}.
\end{equation}
Moreover, define $ Q_w(k)=M(w)+P(w)(ik)^{-1}$ for $k\in\mathbb Z\backslash\{ 0\}$ and $Q_w (0)=M(w)$, then
\begin{equation*}
    \sup_{k\in\mathbb Z}|Q_w(k)|\le C_R,
    \qquad
    \sup_{k\in\mathbb Z}|Q_w(k)-Q_v(k)|
    \le C_R\|w-v\|_{H^\gamma}.
\end{equation*}
Let $g(z)=1-z-\e^{-z}$. For $0<\Tau\le1$, the following estimate holds
\[
  \sup_{k\in\mathbb Z}
    \left|g\big(2i\mu\Tau Q_w(k)\big)\right|
    \le C_R\Tau^2,~\sup_{k\in\mathbb Z}
    \left|
    g\big(2i\mu\Tau Q_w(k)\big)
    -
    g\big(2i\mu\Tau Q_v(k)\big)
    \right|
    \le C_R\Tau^2\|w-v\|_{H^\gamma}.
\]
\end{lem}

\begin{proof}
Since $M(w)-M(v)
=
\Pi_0\big((w-v)\bar w+v(\bar w-\bar v)\big),$
we obtain
\begin{equation*}
    |M(w)-M(v)|
\le C_R\|w-v\|_{H^\gamma}.
\end{equation*}
Similarly, $P(w)-P(v)
=
\Pi_0\big((w-v)\partial_x\bar w
+
v\partial_x(\bar w-\bar v)\big).$ Since $\gamma>3/2$, we obtain~\eqref{eqn:MP_stab}.
The estimates for \(Q_w\) and \(Q_w-Q_v\) follow immediately from $|(ik)^{-1}|\le1,$ when $k\neq0.$

It remains to estimate $g$. Since $g(0)=g'(0)=0,$ and $Q_w(k)$ is uniformly bounded on $B_{R,\gamma}$, Taylor's formula gives $\left|g\big(2i\mu\Tau Q_w(k)\big)\right|
\le C_R\Tau^2.$
Moreover, $g'(z)=-1+\e^{-z},$ so for \(|z|\le C_R\Tau\) we have $|g'(z)|\le C_R\Tau.$ By the mean value theorem,
\[
\begin{aligned}
 \left|
g\big(2i\mu\Tau Q_w(k)\big)
-
g\big(2i\mu\Tau Q_v(k)\big)
\right|   \le
C_R\Tau
\left|
2i\mu\Tau \big(Q_w(k)-Q_v(k)\big)
\right|
\le
C_R\Tau^2\|w-v\|_{H^\gamma}.
\end{aligned}
\]
This completes the proof of the lemma.
\end{proof}

\begin{lem}\label{lem:cubic}
With the CP~\eqref{eqn:cp_2} and the FP~\eqref{eqn:fp_} applied to problem~\eqref{eqn:cubic_1d}, Assumption~\ref{assum:cp} holds. 
\end{lem}

\begin{proof}
We first prove the stability condition~(i). We first define 
\begin{equation*}
    \begin{aligned}
        A(s;w) &=\e^{i(\Tau-s)\partial_x^2} \partial_x^{-1} [\e^{is\partial_x^2} w \cdot \partial_x^{-1} (|\e^{is\partial_x^2} w|^2)] \\
        B(s;w)&= -\frac{1}{2} \e^{i(\Tau-s) \partial_x^2} \partial_x^{-2}[(\e^{-is\partial_x^2}\bar{w})\cdot \e^{is\partial_x^2} (w^2)]. 
    \end{aligned}
\end{equation*}
The CP~\eqref{eqn:cp_2} can be represented as 
\begin{equation*}
\begin{aligned}
    \Phi_{c,2} ( \Tau;w ) &=\e^{i\Tau \partial_{x}^{2}}( \e^{-2i\mu\Tau (M(w)+P(w)\partial_{x}^{-1}) }-1)w-i\mu \Tau \Pi_{0} \left( |w|^{2}w \right) +2i\mu \Tau M (w)\Pi_{0} \left( w \right)\\
    &\quad + \mu B(\Tau,w) - \mu B(0,w) + \mu A(\Tau,w) - \mu A(0,w).
\end{aligned}
\end{equation*}
By the fundamental theorem, we obtain
\begin{align*}
    A(\Tau;w) - A(0;w) &= -2i\int_0^{\Tau} \e^{i(\Tau-s)\partial_x^2} \partial_x^{-1} [(\e^{-is\partial_x^2} \partial_x\bar{w})\cdot (\e^{is\partial_x^2} w)^2]~\d s\\
    &\quad -2 i \int_0^{\Tau} \e^{i(\Tau-s)\partial_x^2}\partial_x^{-1} [|\e^{is\partial_x^2}w|^2 \e^{is\partial_x^2} \partial_x w]\, \d s\\
    &\quad +2i\Tau P(w)\e^{i\Tau \partial_x^2} \partial_x^{-1} w+2i\Tau M(w)(\e^{i\Tau \partial_x^2}w -\Pi_0 (w)),\\
    B(\Tau;w) - B(0;w)&= i \int_0^{\Tau} \e^{i(\Tau-s)\partial_x^2}\partial_x^{-1} [(\e^{-is\partial_x^2}\partial_x \bar{w})\cdot \e^{is\partial_x^2} (w^2)]\,\d s.
\end{align*}
Combined with 
\begin{equation*}
\begin{aligned}
    (\e^{-2i\mu\Tau (M(w)+P(w)\partial_{x}^{-1}) }-1)w &= [\e^{-2i\mu\Tau (M(w)+P(w)\partial_{x}^{-1}) }-1+2i\mu\Tau (M(w)+P(w)\partial_{x}^{-1})]w \\
    &\quad -2i\mu\Tau (M(w)+P(w)\partial_{x}^{-1})w,
    \end{aligned}
\end{equation*}
we obtain
\begin{equation*}
\begin{aligned}
    &\Phi_{c,2} (\Tau;w) = -i\mu \Tau \Pi_0 (|w|^2w) \\
    & -2i\mu \int_0^{\Tau} \e^{i(\Tau -s)\partial_x^2} \partial_x^{-1}[(\e^{-is\partial_x^2}\partial_x \bar{w})\cdot (\e^{is\partial_x^2}w)^2]\, \d s\\
    & -2i\mu \int_0^{\Tau} \e^{i (\Tau -s)\partial_x^2} \partial_x^{-1} [|\e^{is\partial_x^2}w|^2 \e^{is\partial_x^2}\partial_x w]\, \d s\\
    & + i\mu \int_0^{\Tau} \e^{i(\Tau -s)\partial_x^2} \partial_x^{-1} [(\e^{-is\partial_x^2}\partial_x\bar{w})\cdot \e^{is\partial_x^2} (w^2)]\, \d s\\
    & + \e^{i\Tau \partial_x^2} [\e^{-2i\mu\Tau (M(w)+P(w)\partial_{x}^{-1}) }-1+2i\mu\Tau (M(w)+P(w)\partial_{x}^{-1})]w\\
    &=:  {\rm R}_{3,1} (\Tau;w) +  {\rm R}_{3,2} (\Tau;w) + {\rm R}_{3,3}(\Tau;w) + {\rm R}_{3,4} (\Tau;w) + {\rm R}_{3,5}(\Tau;w).
\end{aligned}
\end{equation*}
${\rm R}_{3,1}(\Tau;w)-{\rm R}_{3,1}(\Tau;v)$ corresponds to $D_3(1;w,v)$ in~\eqref{eqn:lem2}. By Lemma~\ref{lem:Dm} and the boundedness of $\Pi_0$, we obtain
\begin{equation*}
    \|{\rm R}_{3,1} (\Tau;w) - {\rm R}_{3,1}(\Tau;v)\|_{H^\gamma}\leq C_R\Tau \|w-v\|_{H^\gamma}.
\end{equation*}
By Sobolev inequalities, the algebra property of $H^{\gamma-1}$ and~\eqref{eqn:bilinear}, we have that for any $\gamma> 3/2$,
\begin{align*}
     &\|{\rm R}_{3,2} (\Tau;w) - {\rm R}_{3,2}(\Tau;v)\|_{H^\gamma} \\
     &\leq C\int_0^{\Tau} \| (\e^{-is\partial_x^2}\partial_x \bar{w})\cdot (\e^{is\partial_x^2}w)^2- (\e^{-is\partial_x^2}\partial_x \bar{v})\cdot (\e^{is\partial_x^2}v)^2\|_{H^{\gamma-1}}\, \d s\\
     &\leq C_R \Tau \|w-v\|_{H^\gamma}.
\end{align*}
The same argument applies to ${\rm R}_{3,3}$ and ${\rm R}_{3,4}$, we obtain
\begin{equation*}
    \|{\rm R}_{3,j} (\Tau;w) - {\rm R}_{3,j}(\Tau;v)\|_{H^\gamma}\leq C_R\Tau \|w-v\|_{H^\gamma}
,\quad j=3,4.
\end{equation*}
It remains to estimate ${\rm R}_{3,5}$. For $k\in \mathbb{Z}$, let $Q_w (k)$ and $g$ be defined as in Lemma~\ref{lem:MP_stability}. Then 
\begin{align*}
&\mathscr F\big(R_{3,5}(\Tau;w)-R_{3,5}(\Tau;v)\big)(k) \\
&=
-\e^{-i\Tau k^2}
g\big(2i\mu\Tau Q_w(k)\big)(\hat w_k-\hat v_k) 
-\e^{-i\Tau k^2}
\left[
g\big(2i\mu\Tau Q_w(k)\big)
-
g\big(2i\mu\Tau Q_v(k)\big)
\right]\hat v_k .
\end{align*}
By Lemma~\ref{lem:MP_stability}, we obtain
\begin{equation*}
    \begin{aligned}
\|R_{3,5}(\Tau;w)-R_{3,5}(\Tau;v)\|_{H^\gamma}
&\le
C_R\Tau^2\|w-v\|_{H^\gamma}
+
C_R\Tau^2\|w-v\|_{H^\gamma}\|v\|_{H^\gamma}  \\
&\le
C_R\Tau^2\|w-v\|_{H^\gamma} \le
C_R\Tau\|w-v\|_{H^\gamma}.
\end{aligned}
\end{equation*}
Therefore, the stability condition~(i) holds.

We now turn to Condition~(ii). The expression can be split into five terms:
    \begin{align*}
&(\F_{\Tau}  -\G_{\Tau})(w_0) - (\F_{\Tau}  -\G_{\Tau})(v_0) \\
&= {\rm R}_{4,1}(\Tau; w_0,v_0) + {\rm R}_{4,2}(\Tau;w_0,v_0)  + {\rm R}_{4,3}(\Tau;w_0,v_0)  + {\rm R}_{4,4}(\Tau;w_0,v_0)  + {\rm R}_{4,5}(\Tau;w_0,v_0),   
\end{align*}
where each \({\rm R}_{4,i}~(i=1,\dotsc,5)\) is defined as follows:
\begin{equation}\label{eqn:R_cubic}
\begin{aligned}
{\rm R}_{4,1}(\Tau;w_0,v_0)
&= -i\mu \int_0^{\Tau} \e^{i(\Tau-s)\partial_x^2}
\Big[
\big(|w(s)|^2w(s)-|\e^{is\partial_x^2}w_0|^2\e^{is\partial_x^2}w_0\big)  \\
&\qquad\qquad
-\big(|v(s)|^2v(s)-|\e^{is\partial_x^2}v_0|^2\e^{is\partial_x^2}v_0\big)
\Big]\,\d s, \\
\mathscr F({\rm R}_{4,2}(\Tau;w,v))(0)
&=  -i\mu
\sum_{k_1+k_2+k_3=0}
\int_0^{\Tau}
\big(\e^{is(k_1^2-k_2^2-k_3^2)}-1\big)\,\d s  \\
&\qquad\qquad\times
\big(
\widehat{\bar w}_{k_1}\hat w_{k_2}\hat w_{k_3}
-
\widehat{\bar v}_{k_1}\hat v_{k_2}\hat v_{k_3}
\big), \\
\mathscr F({\rm R}_{4,3}(\Tau;w,v))(0)
& = 
g\big(2i\mu\Tau M(w)\big)\hat w_0
-
g\big(2i\mu\Tau M(v)\big)\hat v_0,~g(z)=1-z-\e^{-z},\\
\mathscr F({\rm R}_{4,4}(\Tau;w,v))(k)
& =  i\mu \e^{-i\Tau k^2}
\sum_{k=k_1+k_2+k_3}
\int_0^{\Tau}
\frac{k_1}{k}
\big(\e^{2isk_2k_3}-1\big)\e^{2iskk_1}\,\d s  \\
&\qquad\qquad\times
\big(
\widehat{\bar w}_{k_1}\hat w_{k_2}\hat w_{k_3}
-
\widehat{\bar v}_{k_1}\hat v_{k_2}\hat v_{k_3}
\big),
\qquad k\neq0, \\
\mathscr F({\rm R}_{4,5}(\Tau;w,v))(k)
&= 
\e^{-i\Tau k^2}
g\Big(2i\mu\Tau\big(M(w)+P(w)(ik)^{-1}\big)\Big)\hat w_k  \\
&\quad -
\e^{-i\Tau k^2}
g\Big(2i\mu\Tau\big(M(v)+P(v)(ik)^{-1}\big)\Big)\hat v_k,
\qquad k\neq0, \\
\mathscr F({\rm R}_{4,2}(\Tau;w,v))(k)
&= 0,~\mathscr F({\rm R}_{4,3}(\Tau;w,v))(k)
=0,\qquad k\neq0, \\ 
\mathscr F({\rm R}_{4,4}(\Tau;w,v))(0)
&= 0,~\mathscr F({\rm R}_{4,5}(\Tau;w,v))(0)=0.
\end{aligned}
\end{equation}
where $w(s)$ and $v(s)$ in ${\rm R}_{4,1}$ are the exact solutions with initial data $w_0$ and $v_0$ at time $s$, respectively. Following the proof of Lemma~\ref{lem:quadratic}, and $\gamma >3/2$, we obtain 
\begin{equation*}
    \|{\rm R}_{4,1} (\Tau;w_0,v_0)\|_{H^\gamma} \leq C \Tau^2  \|w_0-v_0\|_{H^\gamma}.
\end{equation*}
For ${\rm R}_{4,2}$, the relation $k_1 =-k_2-k_3$ yields
\begin{equation*}
\Big|\int_{0}^{\Tau} \left( \e^{is\left( k_{1}^{2}-k_{2}^{2}-k_{3}^{2} \right)}-1 \right)\, \d s\Big|
 \le \int_{0}^{\Tau} 2s |k_{2}k_{3}| \d s \le  \Tau^{2} |k_{2}k_{3}|.
\end{equation*}
Since ${\rm R}_{4,2}$ is supported only at the zero Fourier mode, we have $\|{\rm R}_{4,2} (\Tau;w,v)\|_{H^\gamma} \leq C|\mathscr{F}({\rm {R}_{4,2}}(\Tau;w,v))(0)|$. Applying~\eqref{eqn:lem1} to ${\rm R}_{4,2}$ then gives $\|{\rm R}_{4,2} (\Tau;w,v)\|_{H^\gamma} \leq C \Tau^2 \|w-v\|_{H^\gamma}.$ For ${\rm R}_{4,3}$, recall that it is supported only at the zero mode, hence 
\begin{align*}
\mathscr F(R_{4,3}(\Tau;w,v))(0)
=
g\big(2i\mu\Tau M(w)\big)(\hat w_0-\hat v_0) +
\left[
g\big(2i\mu\Tau M(w)\big)
-
g\big(2i\mu\Tau M(v)\big)
\right]\hat v_0.
\end{align*}
Since $g(0)=g'(0)=0$, we obtain $|g(2i\mu\Tau M(w))|\leq C_R \Tau^2$. By the mean value theorem, we obtain $g\big(2i\mu\Tau M(w)\big)
-
g\big(2i\mu\Tau M(v)\big) = 2i\mu \Tau g'(\xi)(M(w) - M(v))$ for some $|\xi| \leq C_R \Tau$. Since $R_{4,3} (\Tau;w,v)$ is only supported at zero Fourier mode, $\| {\rm{R}}_{4,3} (\Tau; w,v)\|_{H^r} \leq C|\mathscr{F}({\rm {R}_{4,3}}(\Tau;w,v))(0)|$. By Lemma~\ref{lem:MP_stability}, we obtain
\begin{equation*}
\begin{aligned}
\|{\rm R}_{4,3}(\Tau;w,v)\|_{H^\gamma}
&\le
C_R\Tau^2\|w-v\|_{H^\gamma}
+
C_R\Tau^2\|w-v\|_{H^\gamma}\|v\|_{H^\gamma} \\
&\le
C_R\Tau^2\|w-v\|_{H^\gamma}.
\end{aligned}
\end{equation*}
For ${\rm R}_{4,4}(\Tau;w,v)$, the bound $|\e^{2isk_2 k_3}-1| \leq 2\Tau |k_2 k_3|$ holds for $s\in [0,\Tau]$. The multiplier obtained from the oscillatory remainder is exactly of the form $k^{-1}k_1k_2k_3$, and therefore \eqref{eqn:lem3} applies:
\begin{equation*}
    \| {\rm R}_{4,4} (\Tau;w,v) \|_{H^\gamma} \leq C \Tau^2 \|w - v \|_{H^\gamma}.
\end{equation*}
It remains to estimate ${\rm R}_{4,5}$. For $k\neq0$,
\begin{align*}
\mathscr F({\rm R}_{4,5}(\Tau;w,v))(k)
&=
\e^{-i\Tau k^2}
g\big(2i\mu\Tau Q_w(k)\big)(\hat w_k-\hat v_k) \\
&\quad+
\e^{-i\Tau k^2}
\left[
g\big(2i\mu\Tau Q_w(k)\big)
-
g\big(2i\mu\Tau Q_v(k)\big)
\right]\hat v_k .
\end{align*}
Using Lemma~\ref{lem:MP_stability} again and $|\e^{-i\Tau k^2}|=1$, we obtain
\begin{equation*}
    \begin{aligned}
\|{\rm R}_{4,5}(\Tau;w,v)\|_{H^\gamma}
&\le
C_R\Tau^2\|w-v\|_{H^\gamma}
+
C_R\Tau^2\|w-v\|_{H^\gamma}\|v\|_{H^\gamma}  \\
&\le
C_R\Tau^2\|w-v\|_{H^\gamma}.
\end{aligned}
\end{equation*}
Combining these estimates establishes Condition~(ii).

\end{proof}

\begin{remark}\label{rmk:CP_c2}
In the original scheme proposed in~\cite[(Eq.~(1.11))]{MR4405493}, the local truncation error is analyzed for $(\F_{\Tau}-\G_{\Tau})(u(T_n))$, where $u(T_n)$ denotes the exact solution of~\eqref{eqn:cubic_1d}. Consequently, the quantities involved are $M(u(T_n))$ and $P(u(T_n))$. By the conservation property, one can further use $M(u(T_n))=M(u_0)$ and $P(u(T_n))=P(u_0)$ in~\eqref{eqn:cp_2}. In the parareal setting, however, we evaluate $(\F_{\Tau}-\G_{\Tau})(U_n^k)$ instead, and hence use $M(U_n^k)$ and $P(U_n^k)$ in $\G(U_n^k)$ as defined in~\eqref{eqn:cp_2}. We cannot further replace these quantities by $M_0$ and $P_0$, since in general $M(U_n^k)\neq M_0$ and $P(U_n^k)\neq P_0$.
\end{remark}

As a direct result of Lemma \ref{lem:cubic} and Theorem \ref{assum:cp}, we have the following corollary for the convergence of the proposed exponential low-regularity parareal algorithm.

\begin{coro}\label{cor:42_conv}
Assume that the problem~\eqref{eqn:cubic_1d} admits an exact solution
$u\in C([0,T];H^r(\mathbb{T}))$ with $r>3/2$. Let $U_n^k$ denote the parareal approximation obtained by using the CP~\eqref{eqn:cp_}, augmented by the nonhomogeneous term~\eqref{eqn:cp_2}, together with the exact FP~\eqref{eqn:fp_}, at the $k$-th iteration and the $n$-th coarse grid point, and let $U_n=u(T_n)$ be the exact solution to the problem~\eqref{eqn:cubic_1d}. Then there exist constants $\Tau_0$, $C_1$, and $C_2$, independent of $n$ and $k$, such that, for any $0<\Tau\leq \Tau_0$,
\begin{equation*}
\max_{1\le n\le N_c} \|U_n^k-U_n\|_{H^r}
\leq C_1(C_2\Tau)^k .
\end{equation*}
\end{coro}

\section{Numerical experiment}\label{sec:numerical}
In this section, we test the exponential low-regularity parareal method with different ELRIs on various NLSs. We also compare with four classical CPs, ERK1, ERK3, Lie splitting and Strang splitting. In the numerical experiments, we use a standard Fourier pseudospectral method for the space discretization where we choose the Fourier mode $K=2^{10}$ on $\mathbb{T}$ or $K_x = K_y = 2^9$ on $\mathbb{T}^2$. 

To illustrate the theoretical results in Section~\ref{sec:general case}, we first test the parareal algorithm for the one-dimensional cubic NLS. In addition, we present numerical results for the parareal method with the ELRI~\eqref{eqn:LRI_KSRN} applied to the cubic Schr\"{o}dinger equation on $\mathbb{T}^2$ and the quintic Schr\"{o}dinger equation on $\mathbb{T}$, which lie beyond the theoretical framework developed in Section~\ref{sec:general case}. Nevertheless, the numerical results indicate that the exponential low-regularity parareal method outperforms the classical one even when the initial data are smooth.

When demonstrating the convergence order stated in Theorem~\ref{thm:general_conv}, we initialize the parareal with the given initial data, cf. Figure~\ref{fig:C_NLSE_order}. This ensures that the constant $C_1=\sup_{0\leq i\leq N_c}\|U_i-U_0\|_{H^r}$ in~\eqref{eqn:c1 c2} remains independent of the coarse time step $\Tau$. In all other cases, parareal methods are initialized with the ELRI in that example, so that all parareal errors start from the same point, enabling a clear comparison. All the $L^2$ errors measured in this section are $\max_{0\leq \leq n \leq N_c}\|E_n^k\|_{L^2}$.

\subsection{Numerical results for cubic NLS}
First, we test the one-dimensional cubic NLS \eqref{eqn:cubic_1d} on $\mathbb{T}$. The error estimates are derived under the regularity assumption
$u_0\in H^{\frac{3}{2}+}(\mathbb{T})$. We begin with the borderline low-regularity case
\begin{equation*}
u_0(x)=\frac{\pi-|x-\pi|}{2}\in H^{\frac{3}{2}-}(\mathbb{T}).
\end{equation*}
In the computation, the ELRI~\eqref{eqn:cp_3} is used as the FP with fine time step $\tau=10^{-4}$ and final time $T=1$.
Figure~\ref{fig:C_NLSE_order} illustrates the results of Lemma~\ref{lem:cubic} and Theorem~\ref{thm:general_conv}. The same three CPs as before are tested: the ELRI~\eqref{eqn:cp_3}, Strang splitting, and ERK3. The parareal method with the ELRI achieves the desired order predicted by Theorem~\ref{thm:general_conv}, whereas the classical schemes again stagnate after the first iteration.

\begin{figure}[htbp!]
    \centering
\includegraphics[width=0.98\textwidth,trim={0.8cm 0.8cm 0.5cm 0.5cm},clip]{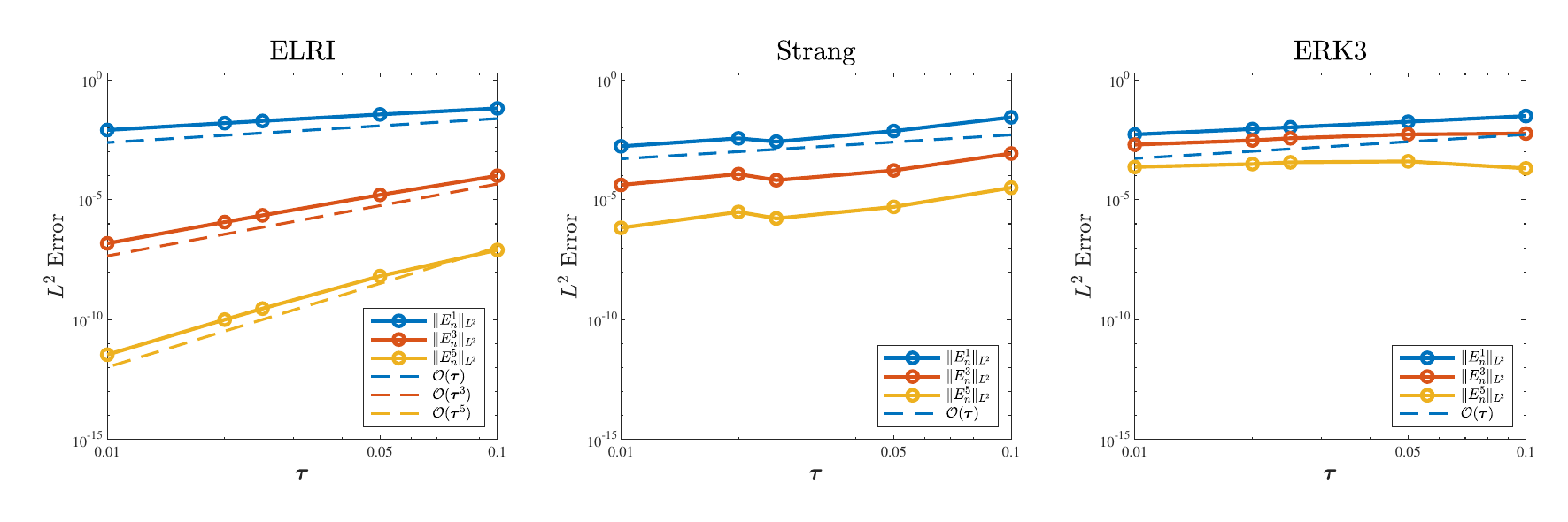}
\caption{$L^2$ error versus   time step $\Tau$ for the parareal applied to cubic NLS  \eqref{eqn:cubic_1d} ($H^{\frac32-}$ initial data, $\mu = 1$) with three CPs: ELRI~\eqref{eqn:cp_3} (left), Strang splitting (middle), and ERK3 (right).}
    \label{fig:C_NLSE_order}
\end{figure}

Furthermore, we test the convergence of the parareal method for the following two initial conditions:
\begin{align*}
u_0(x) &= 0.5\bigl(\min\{x,2\pi-x\}\bigr)^{1.01}\in H^{3/2-}(\mathbb{T}), \\
u_0(x) &= |x-\pi|^{1/2}\sin(x-\pi)\in H^{2-}(\mathbb{T}),
\end{align*}
using five different CPs. We present the convergence histories of the parareal iterations with $T=5$ throughout, taking $\mu=1$ in Figure~\ref{fig:C_NLSE_H4} and $\mu=-1$ in Figure~\ref{fig:C_NLSE_H4_f}. The ELRI~\eqref{eqn:cp_3} is employed as the FP with $\tau = 5\times 10^{-4}$ in both cases. 

We first consider the defocusing case. For the $H^{3/2}$ initial data with $\Tau=0.05, 0.02, 0.01$, and for the $H^2$ initial data with $\Tau=0.05$, the first parareal iteration using the ERK1 solver as the CP exhibits numerical blow-up. For both initial conditions, the exponential low-regularity parareal algorithm, with the scheme~\eqref{eqn:cp_3} used as the CP, significantly outperforms all classical solvers, regardless of their consistency orders.
The convergence behavior of most tested parareals clearly exhibits  two regimes. In the first regime, corresponding to the first few iterations, the parareal errors decrease rapidly. This depends on the regularity of the initial value~\cite{MR3033060}.  
In the second regime, the exponential low-regularity parareal shows rapid linear convergence as predicted in Theorem~\ref{thm:general_conv}, while the errors of the other schemes initially stagnate or even increase before eventually converging superlinearly. The performance of ERK3 is noticeably better than that of ERK1, but it remains far from satisfactory. These results indicate that ERK schemes are unsuitable for problems with limited regularity. It is worth noting that Lie and Strang splitting preserve the mass of the numerical solution, which likely accounts for their improved behavior. Finally, although the Crank--Nicolson (CN) scheme is generally a competitive solver for dispersive equations such as the NLS~\eqref{eqn:NLSE}, the parareal method with CN used as the coarse propagator suffers from numerical blow-up in all tested configurations. This indicates that CN is not suitable within the parareal framework, as is well known for diffusive problems.

For the focusing case shown in Figure~\ref{fig:C_NLSE_H4_f}, the exponential low-regularity parareal method again significantly outperforms the other schemes, exhibiting even greater advantages than in the defocusing case. The convergence still displays the same two-regime behavior. For the $H^{3/2}$ initial data, however, the splitting methods perform worse than ERK3.

\begin{figure}[htbp!]
    \centering
\includegraphics[width=0.98\textwidth,trim={0.8cm 2.2cm 0.5cm 0.5cm},clip]{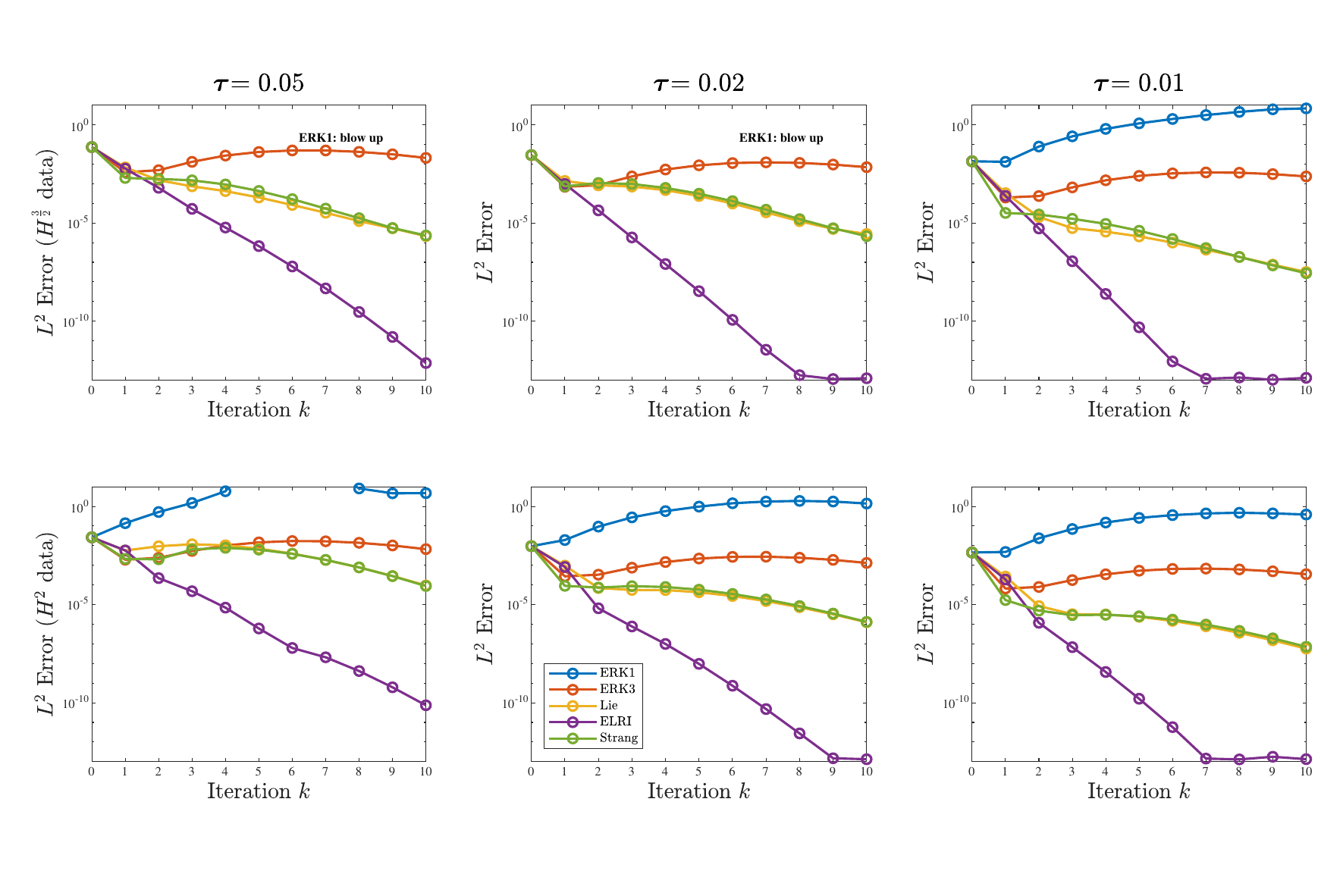}
    \caption{$L^2$ errors of parareal iterations for 1-D cubic NLS with $H^{3/2-}$ and $H^{2-}$ initial data, $\mu = 1$.  }
    \label{fig:C_NLSE_H4}
\end{figure}

\begin{figure}[htbp!]
    \centering
\includegraphics[width=0.98\textwidth,trim={0.8cm 2.2cm 0.5cm 0.5cm},clip]{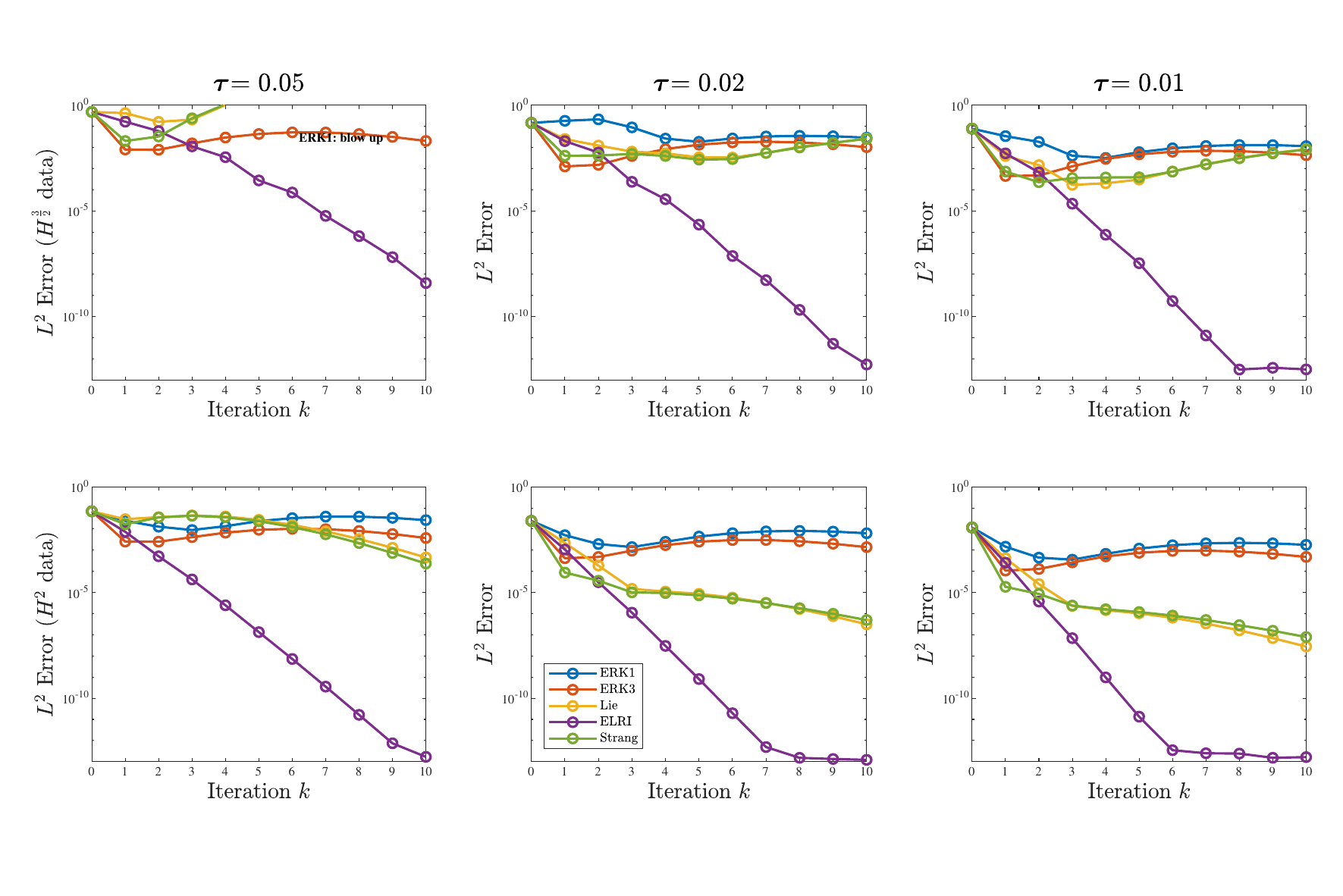}
    \caption{ $L^2$  errors of parareal iterations for 1-D cubic NLS with $H^{3/2-}$ and $H^{2-}$ initial data, {$\mu = -1$}.   }
    \label{fig:C_NLSE_H4_f}
\end{figure}

Next, we present numerical results for the cubic NLS~\eqref{eqn:NLSE} on $\mathbb{T}^2$. Here, the ELRI~\eqref{eqn:LRI_KSRN} is used as the CP in the parareal algorithm. We note that the convergence theory for the parareal method in this setting remains open, since the scheme may not satisfy condition~(ii) in Assumption~\ref{assum:cp}. Nevertheless, these preliminary numerical experiments demonstrate its potential and motivate further study.
We set $\mu=-1$ and $T=5$ in~\eqref{eqn:NLSE}, and consider the initial condition
\begin{equation*}
u_0(x,y)=\big(\sin^2 x+\sin^2 y\big)^{1/2}\in H^{2-}(\mathbb{T}^2).
\end{equation*}
The ELRI~\eqref{eqn:LRI_KSRN} is employed as the FP with the fine time step $\tau = 10^{-3}$. Although the classical schemes provide better initial approximations in the first iteration, the error of the exponential low-regularity parareal method becomes consistently smaller than that of the other schemes after the third iteration. Moreover, it continues to decrease in subsequent iterations, whereas the errors of the other schemes eventually increase at some iteration. As $\Tau$ becomes smaller, the low-regularity parareal exhibits improved performance. A rigorous convergence analysis for this setting is left for future work.

\begin{figure}[htbp!]
    \centering
\includegraphics[width=0.98\textwidth,trim={0.8cm 0.8cm 0.5cm 0.5cm},clip]{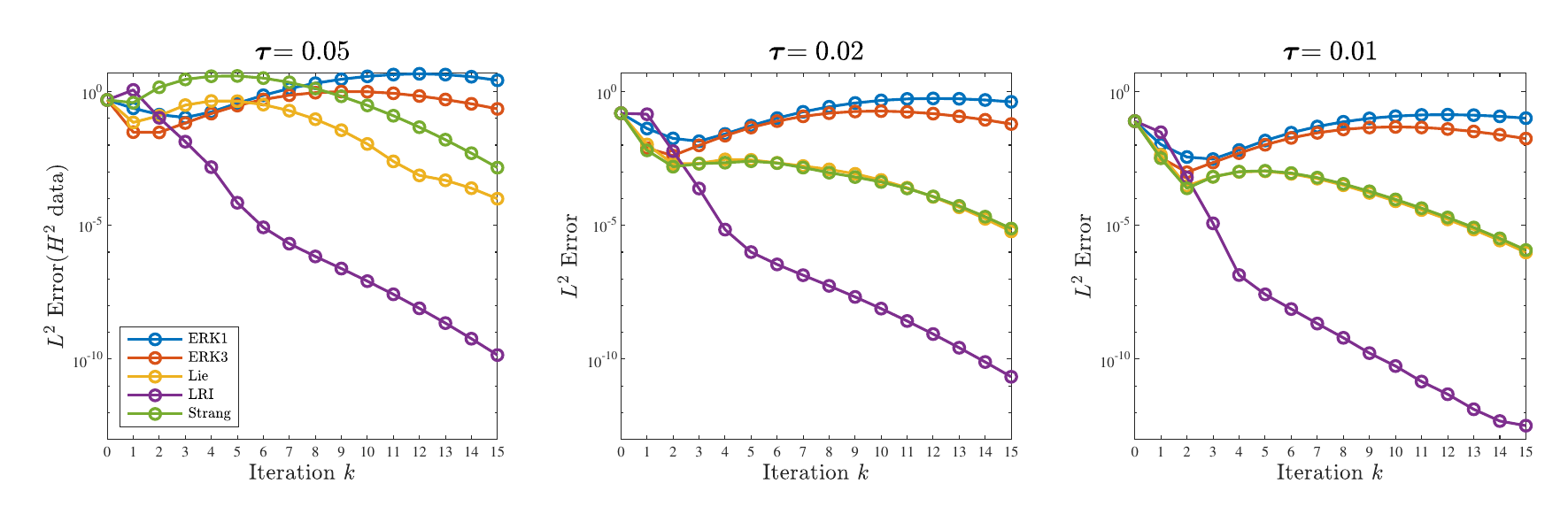}
    \caption{$L^2$ errors of parareal iterations for 2-D cubic NLS with $H^{2-}$ initial data, $\mu = -1$. } 
    \label{fig:NLSE_2D}
\end{figure}

\subsection{The one-dimensional quintic NLS}
We consider the quintic NLS on $\mathbb{T}$, namely, $
F(u,\bar u)=\mu |u|^4u$ and $ d=1$
in~\eqref{eqn:NLSE}. The exponential low-regularity parareal method employs the ELRI~\eqref{eqn:LRI_KSRN} as the CP. We set $\mu=-4$ and $T=16$, and take the initial condition
\[
u_0(x)=\frac{1}{2}|x-\pi|^{1/2}\sin(x-\pi)\in H^{2-}(\mathbb{T}).
\]
The ELRI~\eqref{eqn:LRI_KSRN} is employed as the FP with the fine time step $\tau=10^{-4}$. The exponential low-regularity parareal method exhibits more robust and substantially faster convergence than the alternatives, despite producing the least accurate first-iteration approximation among all CPs. Moreover, its error decreases as $\Tau$ is refined. By contrast, the parareal errors associated with the other CPs either exhibit numerical blow-up or follow a common pattern: they initially increase before eventually decreasing at a superlinear rate, but their overall performance remains poor. A rigorous theoretical investigation of this setting is left for future work.

\begin{figure}[htbp!]
    \centering
\includegraphics[width=0.98\textwidth,trim={0.8cm 0.8cm 0.5cm 0.5cm},clip]{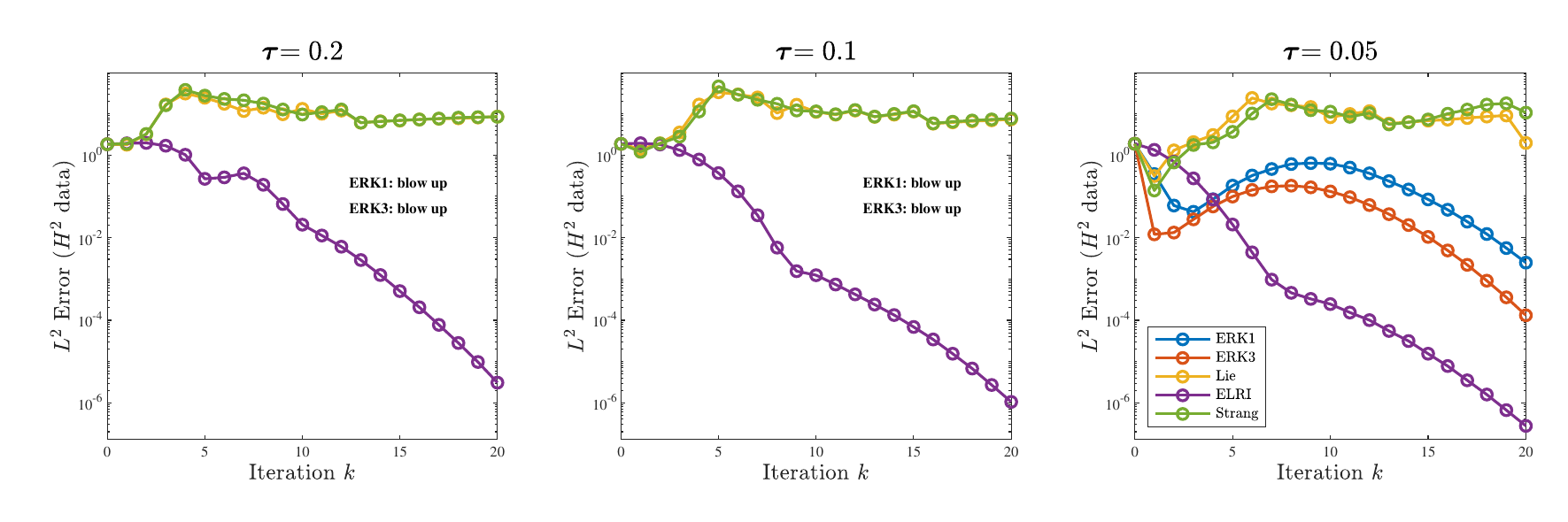}
    \caption{$L^2$ errors of parareal iterations for quintic NLS with $H^{2-}$ initial data, $\mu = -4$, and $T=16$.}
    \label{fig:quintic_SE}
\end{figure}

\section{Concluding remarks}

In this work, we developed and analyzed parareal algorithms for solving NLS equations, assuming an exact fine propagator and focusing on the design of suitable coarse propagators. It is well-known that standard parareal methods may converge slowly or even become unstable for NLS equations, largely due to the absence of damping; moreover, nonlinear interactions can transfer and amplify phase errors across Fourier modes. To address these difficulties, we established a general convergence framework based on stability and local truncation error assumptions for the coarse propagator, while allowing for solutions with limited regularity.
We verified these assumptions for several exponential low-regularity integrators tailored to one-dimensional quadratic and cubic NLS. These integrators are designed to achieve optimal approximation orders without derivative loss, making them particularly suitable as coarse propagators in the parareal framework. With these theoretically justified choices, the resulting exponential low-regularity parareal algorithms are shown to converge linearly, with a contraction factor proportional to the coarse time-step size, even when the underlying solution has only limited regularity.
Extensive numerical experiments for quadratic, cubic, and quintic NLS confirm the rapid convergence of the proposed algorithms. The results also demonstrate their clear advantage over parareal variants based on classical coarse time-stepping schemes, including Lie and Strang splitting methods as well as first- and third-order exponential Runge--Kutta integrators. These findings indicate that exponential low-regularity integrators provide an effective and robust class of coarse propagators for parareal simulation of non-diffusive models.


Several directions remain open for future investigation. First, although the present numerical results demonstrate robust performance for the two-dimensional cubic NLS and the one-dimensional quintic NLS, a rigorous convergence analysis for these cases is still lacking. Extending the current theoretical framework to higher-dimensional problems and higher-order nonlinearities is therefore an important topic for future work. Second, it would be interesting to design and analyze suitable ELRIs for other non-diffusive dispersive or wave-type equations, such as KdV equations, Klein--Gordon equations, and nonlinear wave equations, where the absence of damping and the possible loss of regularity may
pose similar challenges for PinT algorithms. Third, the full-order requirement in condition (ii) of Assumption~\ref{assum:cp} may be relaxed by considering a suboptimal order, or alternatively by formulating the local truncation error assumption in a weaker norm rather than in $H^r$ with $r>d/2$. However, such extensions require a more careful theoretical investigation.

\bibliographystyle{abbrv}
\bibliography{references}

\end{document}


\maketitle

\section{A detailed example}

Here we include some equations and theorem-like environments to show
how these are labeled in a supplement and can be referenced from the
main text.
Consider the following equation:
\begin{equation}
  \label{eq:suppa}
  a^2 + b^2 = c^2.
\end{equation}
You can also reference equations such as \cref{eq:matrices,eq:bb} 
from the main article in this supplement.

\lipsum[100-101]

\begin{theorem}
  An example theorem.
\end{theorem}

\lipsum[102]
 
\begin{lemma}
  An example lemma.
\end{lemma}

\lipsum[103-105]

Here is an example citation: \cite{KoMa14}.

\section[Proof of Thm]{Proof of \cref{thm:bigthm}}
\label{sec:proof}
\lipsum[106-112]

\section{Additional experimental results}
\Cref{tab:foo} shows additional
supporting evidence. 

\begin{table}[htbp]
{\footnotesize
  \caption{Example table}  \label{tab:foo}
\begin{center}
  \begin{tabular}{|c|c|c|} \hline
   Species & \bf Mean & \bf Std.~Dev. \\ \hline
    1 & 3.4 & 1.2 \\
    2 & 5.4 & 0.6 \\ \hline
  \end{tabular}
\end{center}
}
\end{table}

\newpage
\subsubsection{The case $k=0$}
We begin by eliminating the index $k+1$ on the right-hand side of the second equation in \eqref{eqn:coupled} through iterative substitution with respect to $n$:
\begin{align}\label{eqn:k=0}
E_{n+1}^{k+1}&=G_{1}^{2}E_{n-1}^{k+1}+\left( \bar{F}_{1} D_{n}^{k}+\gamma_{0} E_{n}^{k} \right) +G_{1}\left( \bar{F}_{1} D_{n-1}^{k}+\gamma_{0} E_{n-1}^{k} \right) \nonumber \\
&= G_{1}^{n+1} E_{0}^{k+1} + \sum_{i=0}^{n} G_{1}^{i} \left( \bar{F}_1 D_{n-i}^{k} + \gamma_{0} E_{n-i}^{k} \right),
\end{align}
where $E_{0}^{k+1}=U_0^{k+1} - U_0=0$. Setting $k=0$, we obtain:
\begin{align*}
E_{n+1}^{1} &= \sum_{i=0}^{n} G_{1}^{i} \left( F_{1}^{(J)} \left( E_{n-i,-1}^{0} - E_{n-i}^{0} \right) + \gamma_{0} E_{n-i}^{0} \right) \\
&= \sum_{i=0}^{n} G_{1}^{i} F_{1}^{(J)} E_{n-i,-1}^{0} + \sum_{i=0}^{n} G_{1}^{i} \left( \gamma_{0} - F_{1}^{(J)} \right) E_{n-i}^{0}.
\end{align*}
By taking the $L^2$ norm on both sides, we derive the error bound for the case $k=0$,
\begin{equation*}
\| E_{n+1}^{1}\|_{L^{2}} \leq \| F_{1}^{\left( J \right)}\|_{H} \sum_{i=0}^{n} \| E_{i,-1}^{0}\|_{L^{2}} +\| \gamma_{0} -F_{1}^{\left( J \right)}\|_{H} \sum_{i=0}^{n} \| E_{i}^{0}\|_{L^{2}},
\end{equation*}
since $\| G_1 \|_{H} \leq 1$.

\subsubsection{The case $k=1$}
We continue with \eqref{eqn:k=0} in the case $k=0$. Then using the first equation in \eqref{eqn:coupled} to eliminate $D_i^k$ on the right-hand side,
\begin{align*}
E_{n+1}^{k+1} &= \gamma_{0} \sum_{i=0}^{n} G_{1}^{i}E_{n-i}^{k} + \bar{F_{1}} \sum_{i=0}^{n} G_{1}^{i}D_{n-i}^{k} \\
&= \gamma_{0} \sum_{i=0}^{n} G_{1}^{i}E_{n-i}^{k} + \bar{F_{1}} \sum_{i=0}^{n-1} G_{1}^{i}\left( hE_{n-i-1}^{k-1} + \gamma_{1} D_{n-i-1}^{k-1} \right),
\end{align*}
where $D_0^k = E_{0,-1}^k - E_0^k=0$. 
Our next target is to further expand $E_i^k$ on the right-hand side. We have the expression for $E_{n-i}^k$,
\begin{align*}
E_{n-i}^{k} &= G_{1} E_{n-1-i}^{k} + \bar{F_{1}} D_{n-1-i}^{k-1} + \gamma_{0} E_{n-1-i}^{k-1} \\
&= G_{1}^{n-i} E_{0}^{k} + \sum_{j=0}^{n-1-i} G_{1}^{j} \left( \gamma_{0} E_{n-1-i-j}^{k-1} + \bar{F_{1}} D_{n-1-i-j}^{k-1} \right),
\end{align*}
with $E_0^k = U_0^k - U_0=0$. 
Then we continue our expression for the Parareal error $E_{n+1}^{k+1}$:
\begin{align}
E_{n+1}^{k+1} &= \gamma_{0} \sum_{i=0}^{n} G_{1}^{i} \sum_{j=0}^{n-1-i} G_{1}^{j} \left( \gamma_{0} E_{n-1-i-j}^{k-1} + \bar{F_{1}} D_{n-1-i-j}^{k-1} \right) \nonumber \\
&\quad + \bar{F_{1}} h \sum_{i=0}^{n-1} G_{1}^{i} E_{n-1-i}^{k-1} + \bar{F_{1}} \gamma_{1} \sum_{i=0}^{n-1} G_{1}^{i} D_{n-1-i}^{k-1} \nonumber \\
&= \gamma_{0}^{2} \sum_{i=0}^{n} \sum_{j=0}^{n-1-i} G_{1}^{i+j} E_{n-1-i-j}^{k-1} + \gamma_{0} \bar{F_{1}} \sum_{i=0}^{n} \sum_{j=0}^{n-1-i} G_{1}^{i+j} D_{n-1-i-j}^{k-1} \nonumber \\
&\quad + \bar{F_{1}} h \sum_{i=0}^{n-1} G_{1}^{i} E_{n-1-i}^{k-1} + \bar{F_{1}} \gamma_{1} \sum_{i=0}^{n-1} G_{1}^{i} D_{n-1-i}^{k-1} \nonumber \\
&= \sum_{m=0}^{n-1} (m+1) G_{1}^{m} \gamma_{0}^{2} E_{n-1-m}^{k-1} + \gamma_{0} \bar{F_{1}} \sum_{m=0}^{n-1} (m+1) G_{1}^{m} D_{n-1-m}^{k-1} \nonumber\\
&\quad + \bar{F_{1}} h \sum_{i=0}^{n-1} G_{1}^{i} E_{n-1-i}^{k-1} + \bar{F_{1}} \gamma_{1} \sum_{i=0}^{n-1} G_{1}^{i} D_{n-1-i}^{k-1} \nonumber\\
&= \sum_{m=0}^{n-1} G_{1}^{m} \left( (m+1) \gamma_{0}^{2} + \bar{F_{1}} h \right) E_{n-1-m}^{k-1}\nonumber \\
&\quad + \sum_{m=0}^{n-1} G_{1}^{m} \bar{F_{1}} \left( (m+1) \gamma_{0} + \gamma_{1} \right) D_{n-1-m}^{k-1}.\label{eqn:k=1}
\end{align}
Note that the right-hand side of the equality above is only related to the iteration $k-1$. Setting $k=1$ into the equality above, we then obtain
\begin{align*}
E_{n+1}^{2} &= \sum_{m=0}^{n-1} G_{1}^{m} \left( (m+1) \gamma_{0}^{2} + \bar{F}_{1} h \right) E_{n-1-m}^{0} \\
&\quad + \sum_{m=0}^{n-1} G_{1}^{m} \bar{F}_{1} \left( (m+1) \gamma_{0} + \gamma_{1} \right) \left( G_{1} G_{2}^{-1} E_{n-1-m,-1}^{0} - E_{n-1-m}^{0} \right) \\
&= \sum_{m=0}^{n-1} G_{1}^{m} \left( (m+1) \gamma_{0} \left( \gamma_{0} - \bar{F}_{1} \right) + \bar{F}_{1} \left( h - \gamma_{1} \right) \right) E_{n-1-m}^{0} \\
&\quad + \sum_{m=0}^{n-1} G_{1}^{m} \bar{F}_{1} \left( (m+1) \gamma_{0} + \gamma_{1} \right) G_{1} G_{2}^{-1} E_{n-1-m,-1}^{0}.
\end{align*}
After taking the $L^2$ norm on both sides, we then obtain
\begin{align*}
\| E_{n+1}^{2}\|_{L^{2}} &\leq \| \sum_{m=0}^{n-1} G_{1}^{m}\left( \left( m+1 \right) \gamma_{0} \left( \gamma_{0} -\bar{F}_{1} \right) +\bar{F}_{1} \left( h-\gamma_{1} \right) \right) E_{n-1-m}^{0}\|_{L^{2}} \\
&\quad +\| \sum_{m=0}^{n-1} G_{1}^{m}\bar{F}_{1} \left( (m+1)\gamma_{0} +\gamma_{1} \right) G_{1}G_{2}^{-1}E_{n-1-m,-1}^{0}\|_{L^{2}}.
\end{align*}
\red{details in the lemma}For the first term above, we have 
\begin{align*}
& \| \sum_{m=0}^{n-1} G_{1}^{m}\left( \left( m+1 \right) \gamma_{0} \left( \gamma_{0} -\bar{F}_{1} \right) +\bar{F}_{1} \left( h-\gamma_{1} \right) \right) E_{n-1-m}^{0}\|_{L^{2}} \\
& \leq \sup_{p\in \mathbb{N}^{+}} \sum_{m=0}^{n-1} |G_{1,p}^{m}\left( \left( m+1 \right) \gamma_{0,p} \left( \gamma_{0,p} -\bar{F}_{1,p} \right) +\bar{F}_{1,p} \left( h_{p}-\gamma_{1,p} \right) \right) |\sum_{i=0}^{n-1} \| E_{i}^{0}\|_{L^{2}}
\end{align*}
Similarly, for the second term, we have
\begin{align*}
& \left\| \sum_{m=0}^{n-1} G_{1}^{m}\bar{F}_{1} \left( (m+1)\gamma_{0} +\gamma_{1} \right) G_{1}G_{2}^{-1}E_{n-1-m,-1}^{0} \right\|_{L^{2}} \\
& \leq \sup_{p\in \mathbb{N}^{+}} \left| \sum_{m=0}^{n-1} G_{1,p}^{m}\bar{F}_{1,p} \left( \left( m+1 \right) \gamma_{0,p} +\gamma_{1,p} \right) G_{1,p}G_{2,p}^{-1} \right| \sum_{i=0}^{n-1} \left\| E_{i,-1}^{0} \right\|_{L^{2}}.
\end{align*}
Then for $k=1$, we have the following error estimate, 
\begin{align*}
\| E_{n+1}^{2} \|_{L^{2}} 
&\leq \sup_{p\in \mathbb{N}^{+}} \sum_{m=0}^{n-1} \left| G_{1,p}^{m}\left( (m+1) \gamma_{0,p} (\gamma_{0,p} -\bar{F}_{1,p}) + \bar{F}_{1,p} (h_{p}-\gamma_{1,p}) \right) \right| \sum_{i=0}^{n-1} \| E_{i}^{0} \|_{L^{2}} \\
&\quad + \sup_{p\in \mathbb{N}^{+}} \left| \sum_{m=0}^{n-1} G_{1,p}^{m}\bar{F}_{1,p} \left( (m+1)\gamma_{0,p} + \gamma_{1,p} \right) G_{1,p}G_{2,p}^{-1} \right| \sum_{i=0}^{n-1} \left\| E_{i,-1}^{0} \right\|_{L^{2}}.
\end{align*}

\subsubsection{The case $k=2$} 
We continue with \eqref{eqn:k=1} in the case $k=1$. Using the coupled equations \eqref{eqn:coupled}, 
\begin{align*}
E_{n+1}^{k+1} &= \sum_{m=0}^{n-1} G_{1}^{m}\left( \left( m+1 \right) \gamma_{0}^{2} + \bar{F}_{1} h \right) E_{n-1-m}^{k-1} \\
&\quad + \sum_{m=0}^{n-2} G_{1}^{m}\bar{F}_{1} \left( \left( m+1 \right) \gamma_{0} + \gamma_{1} \right) \left( \gamma_{1} D_{n-2-m}^{k-2} + h E_{n-2-m}^{k-2} \right).
\end{align*}
We then further expand $E_{n-1-m}^{k-1}$ and set $k=2$ to obtain
\begin{align*}
E_{n+1}^{3} &= \sum_{m=0}^{n-1} \sum_{j=0}^{n-2-m} G_{1}^{m+j}\left( \left( m+1 \right) \gamma_{0}^{2} +\bar{F}_{1} h \right) \left( \bar{F}_{1} D_{n-2-m-j}^{0}+\gamma_{0} E_{n-2-m-j}^{0} \right) \\
&\quad +\sum_{m=0}^{n-2} G_{1}^{m}\bar{F}_{1} \left( \left( m+1 \right) \gamma_{0} +\gamma_{1} \right) \left( \gamma_{1} D_{n-2-m}^{0}+hE_{n-2-m}^{0} \right) \\
&= \sum_{i=0}^{n-2} G_{1}^{i}\left( \sum_{m=0}^{i} \left( \left( m+1 \right) \gamma_{0}^{2} +\bar{F}_{1} h \right) \right) \left( \bar{F}_{1} D_{n-2-i}^{0}+\gamma_{0} E_{n-2-i}^{0} \right) \\
&\quad +\sum_{m=0}^{n-2} G_{1}^{m}\bar{F}_{1} \left( \left( m+1 \right) \gamma_{0} +\gamma_{1} \right) \left( \gamma_{1} D_{n-2-m}^{0}+hE_{n-2-m}^{0} \right) \\
&= \sum_{i=0}^{n-2} G_{1}^{i} \left( \bar{F}_{1} \left( \left( i+1 \right) \left( \frac{\gamma_{0}^{2} \left( i+2 \right)}{2} +\bar{F}_{1} h+\gamma_{0} \gamma_{1} \right) +\gamma_{1}^{2} \right) D_{n-2-i}^{0} \right. \\
&\quad \left. + \left( \left( i+1 \right) \frac{\gamma_{0}^{3} \left( i+2 \right)}{2} +2\left( i+1 \right) \bar{F}_{1} h\gamma_{0} +\bar{F}_{1} h\gamma_{1} \right) E_{n-2-i}^{0} \right).
\end{align*}

\newpage
\section{Linear Convergence for the Linear Parabolic Problem}

\subsection{Preliminaries}
Consider the following linear parabolic problem
\begin{equation*}
		\left \{
		\begin{aligned}
			u' (t)+ A u(t) &= f(t), \quad 0<t<T,\\
			 u(0)&=u_0 .
		\end{aligned}\right .
\end{equation*}
Thanks to Lemma 10.3 in \cite{thomee2007galerkin}, we are able to give the explicit formula of our FPs:
\begin{align*}
&F\left( T_{0}+n\Delta t,T_{0}+(q-1)\Delta t,u_{0},\cdots ,u_{q-1} \right) \\
    &= \sum_{s=0}^{q-1} \beta_{ns} \left( \Delta tA \right) u_{s} + \Delta t\sum_{j=q}^{n} \beta_{n-j} \left( \Delta tA \right) f\left( T_{0}+j\Delta t \right) \\
    &:= \sum_{s=0}^{q-1} F_{s}^{\left( n-q+1 \right)}(\Delta t A) u_{s} + N^{\left( n-q+1 \right)}\left( f \right) \left( T_{0}+(q-1)\Delta t \right),
\end{align*}
where the \(\beta_j(\lambda)\) and \(\beta_{ns}(\lambda)\) are defined by

$$\widetilde{\beta}(s)=\sum_{j = 0}^{\infty}\beta_j(\lambda)\zeta^j := (\widetilde{\alpha}(\zeta)+\lambda)^{-1}, \quad \beta_{ns}(\lambda)=-\sum_{j = q - s}^{q}\beta_{n - s - j}(\lambda)\alpha_j,$$
with $\tilde{\alpha} \left( \zeta \right) =\alpha_{j} \zeta^{j}$, where the $\alpha_j$ are the coefficients defined by $\tilde{\partial_{q}} U^{n}=\frac{1}{\Delta t} \sum_{j=0}^{q} \alpha_{j} U^{n-j}$.

For convenience, we define several rational functions,  
\begin{equation*}
    h_1(z) = \max_{0 \leq i \leq q-1} \left| \sum_{s=0}^{q-1} \Big( F_{s}^{(J-i)}(z) - F_{s}^{(J)}(z) \Big) \right|,~\gamma_1(z) = \max_{0 \leq i \leq q-1}  \left| F_{s}^{(J-i)}(z) - F_{s}^{(J)}(z) \right|,  
\end{equation*}
\begin{equation*}
    \gamma_0 (z) = \sum_{s=0}^{q-1} F_{s}^{\left( J \right)}\left( z \right) -R\left( Jz \right),~C^{(J)} (z)=\max_{1\leq s\leq q-1} |F_{s}^{\left( J \right)}\left( z \right) |.
\end{equation*}

\newpage
Without loss of generality, we consider the linear problem, $u'+ Au=0$ and the BDF2 scheme as FPs in the Parareal algorithm. Then the multi-step Parareal can be expressed as 
\begin{align} 
\left\{
\begin{aligned}
U_{n+1}^{k+1}&=G_{1}\left( U_{n}^{k+1} \right) +F_{1}\left( U_{n,-1}^{k} \right) +F_{2}\left( U_{n}^{k} \right) -G_{1}\left( U_{n}^{k} \right), \\
U_{n+1,-1}^{k+1}&=G_{2}\left( U_{n}^{k+1} \right) +F_{1}^{\prime}\left( U_{n,-1}^{k} \right) +F_{2}'\left( U_{n}^{k} \right) -G_{2}\left( U_{n}^{k} \right).
\end{aligned}
\right.
\end{align}
The exact solution also satisfies \begin{align*}
\left\{
\begin{aligned}
U_{n+1}      &= F_{1}\left( U_{n,-1} \right) + F_{2}\left( U_{n} \right), \\
U_{n+1,-1} &= F_{1}^{\prime}\left( U_{n,-1} \right) + F_{2}^{\prime}\left( U_{n} \right).
\end{aligned}
\right.
\end{align*} 
Thus, we define the Parareal error $E_n^k = U_n^k - U_n$ and $E_{n,-1}^k = U_{n,-1}^k - U_{n,-1}$, which satisfy
\begin{align}
\left\{
\begin{aligned}
E_{n+1}^{k+1}&=G_{1}\left( E_{n}^{k+1} \right) +F_{1}\left( E_{n,-1}^{k} \right) +F_{2}\left( E_{n}^{k} \right) -G_{1}\left( E_{n}^{k} \right), \\
E_{n+1,-1}^{k+1}&=G_{2}\left( E_{n}^{k+1} \right) +F_{1}^{\prime}\left( E_{n,-1}^{k} \right) +F_{2}'\left( E_{n}^{k} \right) -G_{2}\left( E_{n}^{k} \right).
\end{aligned}
\right.
\end{align}
If $G_2$ is invertible, the second equation in \eqref{eqn:para_err} is equal to 
\begin{equation*}
G_{1}G_{2}^{-1}\left( E_{n+1,-1}^{k+1} \right) =G_{1}\left( E_{n}^{k+1} \right) +G_{1}G_{2}^{-1}F_{1}^{\prime}\left( E_{n,-1}^{k} \right) +G_{1}G_{2}^{-1}F_{2}^{\prime}\left( E_{n}^{k} \right) -G_{1}\left( E_{n}^{k} \right).
\end{equation*}
Given that $F_1',~F_2'$, and $G_1 G_2^{-1}$ mutually commute, this can be rewritten as 
\begin{equation*}
    G_{1}G_{2}^{-1}\left( E_{n+1,-1}^{k+1} \right) =G_{1}\left( E_{n}^{k+1} \right) +F_{1}^{\prime}G_{1}G_{2}^{-1}\left( E_{n,-1}^{k} \right) +F_{2}^{\prime}G_{1}G_{2}^{-1}\left( E_{n}^{k} \right) -G_{1}\left( E_{n}^{k} \right).
\end{equation*}
We further introduce one new variable, $D_{n}^{k}=G_{1}G_{2}^{-1}\left( E_{n,-1}^{k} \right) -E_{n}^{k}$, which satisfies
\begin{align}
D_{n+1}^{k+1} &= \left( F_{1}^{\prime}G_{1}G_{2}^{-1}-F_{1} \right) \left( E_{n,-1}^{k} \right) +\left( F_{2}'{ G_{1}G_{2}^{-1}}-F_{2} \right) \left( E_{n}^{k} \right) \nonumber \\
&= \left( F_{1}^{\prime}-F_{1}G_{2}G_{1}^{-1} \right) \left( G_{1}G_{2}^{-1}\left( E_{n,-1}^{k} \right) -E_{n}^{k} \right) \nonumber \\
&\quad +\left( F_{1}^{\prime}+F_{2}^{\prime}G_{1}G_{2}^{-1}-F_{1}G_{2}G_{1}^{-1}-F_{2} \right) \left( E_{n}^{k} \right) \nonumber \\
&:= \gamma_{1} D_{n}^{k}+hE_{n}^{k}.
\end{align}
The Parareal error $E_{n+1}^{k+1}$ can also be related with $D_n^k$ through
\begin{align*}
E_{n+1}^{k+1} &= G_{1}\left( E_{n}^{k+1} \right) + F_{1}G_{2}G_{1}^{-1}\left( G_{1}G_{2}^{-1}\left( E_{n,-1}^{k} \right) - E_{n}^{k} \right) \\
&\quad + \left( F_{1}G_{2}G_{1}^{-1} + F_{2} - G_1 \right) \left( E_{n}^{k} \right) \\
&:= G_{1} E_{n}^{k+1} + \bar{F_1} D_{n}^{k} + \gamma_{0} E_{n}^{k}. 
\end{align*}
Then we derive the coupled equations between $E_n^k$ and $D_n^k$:
\begin{equation}
    \left\{
\begin{aligned}
D_{n+1}^{k+1} &= \gamma_{1} D_{n}^{k} + h E_{n}^{k}, \\
E_{n+1}^{k+1} &= G_{1} E_{n}^{k+1} + \bar{F}_{1} D_{n}^{k} + \gamma_{0} E_{n}^{k}.
\end{aligned}
\right.
\end{equation}
Our further estimate bases on these coupled equations. Now we start from eliminating the index $k+1$ of the right-hand side of the second equation in \eqref{eqn:coupled} by iterating itself on $n$,
\begin{align*}
E_{n+1}^{k+1} &= G_{1}^{2} E_{n-1}^{k+1}  + \sum_{i=0}^{1} G_{1}^{i}\left( \bar{F}_1 D_{n-i}^{k} + \gamma_{0} E_{n-i}^{k} \right) \\
&= G_{1}^{n+1} E_{0}^{k+1}  + \sum_{i=0}^{n} G_{1}^{i}\left( \bar{F}_1 D_{n-i}^{k} + \gamma_{0} E_{n-i}^{k} \right),
\end{align*}
where $E_{0}^{k+1}=U_{0}^{k+1}-U_{0}=0$. Then using the first equation in \eqref{eqn:coupled} to eliminate $D_i^k$ on the right-hand side,
\begin{align*}
E_{n+1}^{k+1} &= \gamma_{0} \sum_{i=0}^{n} G_{1}^{i}E_{n-i}^{k} + \bar{F_{1}} \sum_{i=0}^{n} G_{1}^{i}D_{n-i}^{k} \\
&= \gamma_{0} \sum_{i=0}^{n} G_{1}^{i}E_{n-i}^{k} + \bar{F_{1}} \sum_{i=0}^{n} G_{1}^{i}\left( hE_{n-i-1}^{k-1} + \gamma_{1} D_{n-i-1}^{k-1} \right).
\end{align*}
Our next target is to further expand $E_i^k$ on the right-hand side, because the numerical experiments indicate that we should derive staggered time steps, and if we don't follow this rule, our estimate is not accurate enough. We have the expression for $E_{n-i}^k$,
\begin{align*}
E_{n-i}^{k} &= G_{1} E_{n-1-i}^{k} + \bar{F_{1}} D_{n-1-i}^{k-1} + \gamma_{0} E_{n-1-i}^{k-1} \\
&= G_{1}^{n-i} E_{0}^{k} + \sum_{j=0}^{n-i} G_{1}^{j} \left( \gamma_{0} E_{n-1-i-j}^{k-1} + \bar{F_{1}} D_{n-1-i-j}^{k-1} \right),
\end{align*}
with $E_0^k = U_0^k - U_0=0$. 
Then we continue our expression for the Parareal error $E_{n+1}^{k+1}$:
\begin{align*}
E_{n+1}^{k+1} &= \gamma_{0} \sum_{i=0}^{n} G_{1}^{i} \sum_{j=0}^{n-i} G_{1}^{j} \left( \gamma_{0} E_{n-1-i-j}^{k-1} + \bar{F_{1}} D_{n-1-i-j}^{k-1} \right) \\
&\quad + \bar{F_{1}} h \sum_{i=0}^{n} G_{1}^{i} E_{n-1-i}^{k-1} + \bar{F_{1}} \gamma_{1} \sum_{i=0}^{n} G_{1}^{i} D_{n-1-i}^{k-1} \\
&= \gamma_{0}^{2} \sum_{i=0}^{n} \sum_{j=0}^{n-i} G_{1}^{i+j} E_{n-1-i-j}^{k-1} + \gamma_{0} \bar{F_{1}} \sum_{i=0}^{n} \sum_{j=0}^{n-i} G_{1}^{i+j} D_{n-1-i-j}^{k-1} \\
&\quad + \bar{F_{1}} h \sum_{i=0}^{n} G_{1}^{i} E_{n-1-i}^{k-1} + \bar{F_{1}} \gamma_{1} \sum_{i=0}^{n} G_{1}^{i} D_{n-1-i}^{k-1} \\
&= \gamma_{0} \sum_{m=0}^{n} (m+1) G_{1}^{m} \gamma_{0}^{2} E_{n-1-m}^{k-1} + \gamma_{0} \bar{F_{1}} \sum_{m=0}^{n} (m+1) G_{1}^{m} D_{n-1-m}^{k-1} \\
&\quad + \bar{F_{1}} h \sum_{i=0}^{n} G_{1}^{i} E_{n-1-i}^{k-1} + \bar{F_{1}} \gamma_{1} \sum_{i=0}^{n} G_{1}^{i} D_{n-1-i}^{k-1} \\
&= \sum_{m=0}^{n} G_{1}^{m} \left( (m+1) \gamma_{0}^{2} + \bar{F_{1}} h \right) E_{n-1-m}^{k-1} \\
&\quad + \sum_{m=0}^{n} G_{1}^{m} \bar{F_{1}} \left( (m+1) \gamma_{0} + \gamma_{1} \right) D_{n-1-m}^{k-1}.
\end{align*}
Note that the right-hand side of the equality above is only related to the iteration $k-1$.\red{$12345$} Now we further expand $D_{n-1-m}^{k-1}$ to iteration $k-3$ based on \eqref{eqn:coupled}: 
\begin{align*}
D_{n-1-m}^{k-1} &= h E_{n-2-m}^{k-2} + \gamma_{1} D_{n-2-m}^{k-2} \\
&= h \left( G_{1} E_{n-3-m}^{k-2} + \gamma_{0} E_{n-3-m}^{k-3} + \bar{F_{1}} D_{n-3-m}^{k-3} \right) + \gamma_{1} D_{n-2-m}^{k-2} \\
&= h \sum_{j=0}^{n-3-m} G_{1}^{j} \left( \gamma_{0} E_{n-3-m-j}^{k-3} + \bar{F_{1}} D_{n-3-m-j}^{k-3} \right) + \gamma_{1} D_{n-2-m}^{k-2}.
\end{align*}
Then we further expand $D_{n-2-m}^{k-2}$ to eliminate iteration $k-2$ on the right-hand side of the above equation, 
\begin{align*}
D_{n-1-m}^{k-1} &= h \sum_{j=0}^{n-3-m} G_{1}^{j} \left( \gamma_{0} E_{n-3-m-j}^{k-3} + \bar{F_{1}} D_{n-3-m-j}^{k-3} \right) + \gamma_{1} h E_{n-3-m}^{k-3} + \gamma_{1}^{2} D_{n-3-m}^{k-3} \\
&= h \gamma_{0} \sum_{j=0}^{n-3-m} G_{1}^{j} E_{n-3-m-j}^{k-3} + \gamma_{1} h E_{n-3-m}^{k-3} + h \bar{F_{1}} \sum_{j=0}^{n-3-m} G_{1}^{j} D_{n-3-m-j}^{k-3} + \gamma_{1}^{2} D_{n-3-m}^{k-3}. 
\end{align*}
Finally, we take the expression of $D_{n-1-m}^{k-1}$ back to the Parareal error $E_{n+1}^{k+1}$,  
\begin{align*}
E_{n+1}^{k+1} &= \sum_{m=0}^{n} G_{1}^{m} \left( (m+1) \gamma_{0}^{2} + \bar{F_{1}} h \right) E_{n-1-m}^{k-1} \\
&\quad + \sum_{m=0}^{n} G_{1}^{m} \bar{F_{1}} \left( (m+1) \gamma_{0} + \gamma_{1} \right) h \left( \gamma_{0} \sum_{j=0}^{n-3-m} G_{1}^{j} E_{n-3-m-j}^{k-3} + \gamma_{1} E_{n-3-m}^{k-3} \right) \\
&\quad + \sum_{m=0}^{n} G_{1}^{m} \bar{F_{1}} \left( (m+1) \gamma_{0} + \gamma_{1} \right) \left( h \bar{F_{1}} \sum_{j=0}^{n-3-m} G_{1}^{j} D_{n-3-m-j}^{k-3} + \gamma_{1}^{2} D_{n-3-m}^{k-3} \right).
\end{align*}
Now we take a close look at the equation above. The first row involves $E_i^{k-1}$, the second row involves $E_i^{k-3}$, and the third row involves $D_i^{k-3}$. To analyze the $L^2$ norm of $E_{n+1}^{k+1}$, we consider the $p$-th spectrum of the elliptic operator $A$ and take the inner product with the eigenfunction $\lambda_p$.  
\begin{align*}
(E_{n+1}^{k+1},\lambda_p) &= \sum_{m=0}^{n} G_{1}^{m} \left( (m+1) \gamma_{0}^{2} + \bar{F_{1}} h \right) (E_{n-1-m}^{k-1},\lambda_p)  \\
&\quad + \sum_{m=0}^{n} G_{1}^{m} \bar{F_{1}} \left( (m+1) \gamma_{0} + \gamma_{1} \right) h \left( \gamma_{0} \sum_{j=0}^{n-3-m} G_{1}^{j} (E_{n-3-m-j}^{k-3},\lambda_p) + \gamma_{1} (E_{n-3-m}^{k-3},\lambda_p) \right) \\
&\quad + \sum_{m=0}^{n} G_{1}^{m} \bar{F_{1}} \left( (m+1) \gamma_{0} + \gamma_{1} \right) \left( h \bar{F_{1}} \sum_{j=0}^{n-3-m} G_{1}^{j} (D_{n-3-m-j}^{k-3},\lambda_p) + \gamma_{1}^{2} (D_{n-3-m}^{k-3},\lambda_p)  \right).
\end{align*}
Let $e_n^k=(E_n^k,\lambda_p)$ and $d_n^k = (D_n^k,\lambda_p)$. We omit the explicit dependence of $e_n^k$, $d_n^k$ and all the operators appear above on $p$ for brevity. 
\begin{align*}
e_{n+1}^{k+1} &= \sum_{m=0}^{n} G_{1}^{m} \left( (m+1) \gamma_{0}^{2} + \bar{F_{1}} h \right) e_{n-1-m}^{k-1} \\
&\quad + \sum_{m=0}^{n} G_{1}^{m} \bar{F_{1}} \left( (m+1) \gamma_{0} + \gamma_{1} \right) h \left( \gamma_{0} \sum_{j=0}^{n-3-m} G_{1}^{j} e_{n-3-m-j}^{k-3} + \gamma_{1} e_{n-3-m}^{k-3} \right) \\
&\quad + \sum_{m=0}^{n} G_{1}^{m} \bar{F_{1}} \left( (m+1) \gamma_{0} + \gamma_{1} \right) \left( h \bar{F_{1}} \sum_{j=0}^{n-3-m} G_{1}^{j} d_{n-3-m-j}^{k-3} + \gamma_{1}^{2} d_{n-3-m}^{k-3} \right).
\end{align*}

\textbf{(Term I)} This term represents the primary error component of $E_{n+1}^{k+1}$, as it incorporates the complete error propagation from the $(k-1)$-th iteration. 
\begin{align*}
\left| \sum_{m=0}^{n} G_{1}^{m}\left( (m+1)\gamma_{0}^{2} +\bar{F_{1}} h \right) e_{n-1-m}^{k-1}\right|
&\leq \max_{i} |e_{i}^{k-1}|\sum_{m=0}^{n} |G_{1}|^{m}\left| (m+1) \gamma_{0}^{2} +\bar{F_{1}} h\right| \\
&\leq \gamma_{a,p} \max_{i} |e_{i}^{k-1}|.
\end{align*}
We have to make this bound clear: 
\begin{equation}\label{eqn:g_a}
\gamma_{a,p}:=\sum_{m=0}^{n} |G_{1}|^{m}\left| (m+1) \gamma_{0,p}^{2} +\bar{F_{1}} h_p\right|. 
\end{equation}

\begin{equation*}
\gamma_{a,p} :=\left( \frac{\gamma_{0,p}}{1-|G_{1}\left( Jz_{p} \right) |} \right)^{2} +\frac{|\bar{F}_{1,p} h_{p}|}{1-|G_{1}\left( Jz_{p} \right) |},
\end{equation*}
where $z_p = \Delta t\lambda_p$ and 
\begin{align*}
\gamma_{0,p} &= \left| \left( F_{1}(z_{p}) G_{2}(Jz_{p}) G_{1}^{-1}(Jz_{p}) + F_{2}(z_{p}) \right) - G_{1}(Jz_{p}) \right|, \\
\bar{F}_{1,p} &= |F_{1}(z_{p}) G_{2}(Jz_{p}) G_{1}^{-1}(Jz_{p})|, \\
h_{p} &= \left| F_{1}'(z_{p}) + F_{2}'(z_{p}) G_{1}(Jz_{p})G_2^{-1}(Jz_p) - F_{1}(z_{p}) G_{2}(Jz_{p}) G_{1}^{-1}(Jz_{p}) - F_{2}(z_{p}) \right|.
\end{align*}

\textbf{(Term II)} We bound for the second summation: 
\begin{align*}
&\left| h\gamma_{0} \bar{F_{1}} \sum_{m=0}^{n} \sum_{j=0}^{n-3-m} G_{1}^{m+j}\left( \left( m+1 \right) \gamma_{0} +\gamma_{1} \right) e_{n-3-m-j}^{k-3} \right. \\
&\quad \left. + h\gamma_{1} \bar{F_{1}} \sum_{m=0}^{n} G_{1}^{m}\left( \left( m+1 \right) \gamma_{0} +\gamma_{1} \right) e_{n-3-m}^{k-3} \right| \\
&\leq \max_{i} |e_{i}^{k-3}| \left( |h\gamma_{0} \bar{F_{1}} |\sum_{m=0}^{n} \sum_{j=0}^{n-3-m} |G_{1}|^{m+j}|\left( m+1 \right) \gamma_{0} +\gamma_{1} | \right) \\
&\quad + \max_{i} |e_{i}^{k-3}| \left( |h\gamma_{1} \bar{F_{1}} |\sum_{m=0}^{n} |G_{1}|^{m}|\left( m+1 \right) \gamma_{0} +\gamma_{1} | \right) \\
&\leq \left( \frac{|\bar{F_{1}} h\gamma_{0}^{2} |}{\left( 1-|G_{1}| \right)^{3}} + \frac{|\bar{F_{1}} h\gamma_{0} \gamma_{1} |}{\left( 1-|G_{1}| \right)^{2}} \right) \max_{i} |e_{i}^{k-3}| \\
&\quad + \left( \frac{|\bar{F_{1}} h\gamma_{0} \gamma_{1} |}{\left( 1-|G_{1}| \right)^{2}} + \frac{|\bar{F_{1}} h\gamma_{1}^{2} |}{1-|G_{1}|} \right) \max_{i} |e_{i}^{k-3}| \\
&\leq  \frac{|\bar{F_{1}} h|}{1-|G_{1}|} \left( |\gamma_{1} |+\frac{|\gamma_{0} |}{1-|G_{1}|} \right)^{2} \max_{i} |e_{i}^{k-3}|.
\end{align*}
We should make the bound clear: 
\begin{equation}\label{eqn:g_b}
    \gamma_{b,p} :=\frac{\bar{F}_{1,p} h_{p}}{1-|G_{1}\left( Jz_{p} \right) |} \left( \gamma_{1,p} +\frac{\gamma_{0,p}}{1-|G_{1}\left( Jz_{p} \right) |} \right)^{2},
\end{equation}
where 
\begin{equation*}
  \gamma_{1,p} =|F_{1}^{\prime}(z_{p})-F_{1}(z_{p})G_{2}(Jz_{p})G_{1}^{-1}(Jz_{p})|.
\end{equation*}

\textbf{(Term III)} 
We analyze the final summation term, specifically examining the misalignment-induced error $d_i^{k-3}$. Our approach involves expressing $d_i^{k+1}$ in terms of $d_i^{k-1}$ and $e_i^{k-1}$, which enables us to establish an upper bound for $d_i^{k+1}$. We start from \eqref{eqn:coupled},
\begin{align*}
d_{n+1}^{k+1} &= h e_{n}^{k} + \gamma_{1} d_{n}^{k} \\
&= h \left( G_{1} e_{n-1}^{k} + \gamma_{0} e_{n-1}^{k-1} + \bar{F_{1}} d_{n-1}^{k-1} \right) + \gamma_{1} d_{n}^{k} \\
&= h \sum_{j=0}^{n-1} G_{1}^{j} \left( \gamma_{0} e_{n-1-j}^{k-1} + \bar{F_{1}} d_{n-1-j}^{k-1} \right) + +\gamma_{1} d_{n}^{k}.
\end{align*}
Then we further expand $d_n^k$ with $d_{n}^{k}=he_{n-1}^{k-1}+\gamma_{1} d_{n-1}^{k-1}$,
\begin{align*}
d_{n+1}^{k+1} &= h\sum_{j=0}^{n-1} G_{1}^{j}\left(\gamma_{0} e_{n-1-j}^{k-1} + \bar{F_{1}} d_{n-1-j}^{k-1}\right) + \gamma_{1}\left(he_{n-1}^{k-1} + \gamma_{1} d_{n-1}^{k-1}\right) \\
&= h\gamma_{0} \sum_{j=1}^{n-1} G_{1}^{j}e_{n-1-j}^{k-1} + h(\gamma_{1}+\gamma_0) e_{n-1}^{k-1}+ h\bar{F_{1}} \sum_{j=1}^{n-1} G_{1}^{j}d_{n-1-j}^{k-1} + (h\bar{F}_1+\gamma_{1}^{2}) d_{n-1}^{k-1}.
\end{align*}
We take the absolute norm on both sides and derive the bound for $d_{n+1}^{k+1}$, 
\begin{align*}
    |d_{n+1}^{k+1}| &\leq \left( \frac{|G_{1} h \gamma_{0}|}{1 - |G_{1}|} + |h (\gamma_{0} + \gamma_{1})| \right) \max_{i} |e_{i}^{k-1}| 
+ \left( \frac{|G_{1} h \bar{F}_{1}|}{1 - |G_{1}|} + |h \bar{F}_{1} + \gamma_{1}^{2}| \right) \max_{i} |d_{i}^{k-1}|\\
&:= \gamma_{d,p}\max_{i}|e_i^{k-1}| + \gamma_{e,p} \max_{i} |d_i^{k-1}|,
\end{align*}
where we define the linear bounds, 
\begin{equation}\label{eqn:g_d,d_e}
    \gamma_{d,p} =\frac{|G_{1}\left( Jz_{p} \right) h_{p}\gamma_{0,p} |}{1-|G_{1}\left( Jz_{p} \right) |} +|h_{p}\left( \gamma_{0,p} +\gamma_{1,p} \right) |,~\gamma_{e,p} =\frac{|G_{1}\left( Jz_{p} \right) h_{p}\bar{F}_{1,p} |}{1-|G_{1}\left( Jz_{p} \right) |} +|h_{p}\bar{F}_{1,p} +\gamma_{1,p}^{2} |.
\end{equation}
Then we derive the bound for the last summation, 
\begin{align*}
&\left| \sum_{m=0}^{n} G_{1}^{m}\bar{F}_{1} \left( (m+1)\gamma_{0} + \gamma_{1} \right) \left( h\bar{F}_{1} \sum_{j=0}^{n-3-m} G_{1}^{j}d_{n-3-m-j}^{k-3} + \gamma_{1}^{2} d_{n-3-m}^{k-3} \right) \right| \\
&\leq \left| h\bar{F}_{1}^{2} \sum_{m=0}^{n} \sum_{j=0}^{n-3-m} G_{1}^{m+j} \left( (m+1)\gamma_{0} + \gamma_{1} \right) d_{n-3-m-j}^{k-3} \right| \\
&\quad + \left| \gamma_{1}^{2} \bar{F}_{1} \sum_{m=0}^{n} G_{1}^{m} \left( (m+1)\gamma_{0} + \gamma_{1} \right) d_{n-3-m}^{k-3} \right| \\
&\leq \left( \frac{|h\bar{F}_{1}^{2} \gamma_{0}|}{(1-|G_{1}|)^{3}} + \frac{|h\bar{F}_{1}^{2} \gamma_{1}|}{(1-|G_{1}|)^{2}} \right) \max_{i} |d_{i}^{k-3}|  + \left( \frac{|\bar{F}_{1} \gamma_{0} \gamma_{1}^{2}|}{(1-|G_{1}|)^{2}} + \frac{|\bar{F}_{1} \gamma_{1}^{3}|}{1-|G_{1}|} \right) \max_{i} |d_{i}^{k-3}|\\
&=  \left( \frac{|h\bar{F}_{1}^{2} \gamma_{0}|}{(1-|G_{1}|)^{3}} + \frac{|h\bar{F}_{1}^{2} \gamma_{1}|}{(1-|G_{1}|)^{2}} + \frac{|\bar{F}_{1} \gamma_{0} \gamma_{1}^{2}|}{(1-|G_{1}|)^{2}} + \frac{|\bar{F}_{1} \gamma_{1}^{3}|}{1-|G_{1}|} \right) \max_{i} |d_{i}^{k-3}|.
\end{align*}
Then we clarify the bound, but this is a little bit complicated, 
\begin{equation}\label{eqn:g_c}
    \gamma_{c,p} =\frac{h_{p}\bar{F}_{1,p}^{2} \gamma_{0,p}}{\left( 1-|G_{1}\left( Jz_{p} \right) | \right)^{3}} +\frac{h_{p}\bar{F}_{1,p}^{2} \gamma_{1,p}}{\left( 1-|G_{1}\left( Jz_{p} \right) | \right)^{2}} +\frac{\bar{F}_{1,p} \gamma_{0,p} \gamma_{1,p}^{2}}{\left( 1-|G_{1}\left( Jz_{p} \right) | \right)^{2}} +\frac{\bar{F}_{1,p} \gamma_{1,p}^{3}}{1-|G_{1}\left( Jz_{p} \right) |}.
\end{equation}

Finally, we arrive at the iteration on index $k$, based on the estimates of the three terms: 
\begin{align*}
\left\{
\begin{aligned}
|e_{n+1}^{k+1}| &\leq \gamma_{a,p} \max_{i} |e_{i}^{k-1}| + \gamma_{b,p} \max_{i} |e_{i}^{k-3}| + \gamma_{c,p} \max_{i} |d_{i}^{k-3}|, \\
|d_{n}^{k-1}|   &\leq \gamma_{d,p} \max_{i} |e_{i}^{k-3}| + \gamma_{e,p} \max_{i} |d_{i}^{k-3}|,
\end{aligned}
\right.
\end{align*}
where the left-hand side of the first equation is independent of $n$, then we take the maximum over $n$. We further introduce $x_k=\max_{i}|e_i^{k}|$ and $y_k = \max_{i} |d_i^k|$ and obtain
\begin{align}\label{eqn:bound}
\left\{
\begin{aligned}
x_{k+1} &\leq \gamma_{a,p} x_{k-1} + \gamma_{b,p} x_{k-3} + \gamma_{c,p} y_{k-3}, \\
y_{k-1} &\leq \gamma_{d,p} x_{k-3} + \gamma_{e,p} y_{k-3}.
\end{aligned}
\right.
\end{align}
Next we study the recursion, 
\begin{align*}
x_{k+1} &\leq \gamma_{a,p} x_{k-1}+\gamma_{b,p} x_{k-3}+\gamma_{c,p} \left( \gamma_{d,p} x_{k-5}+\gamma_{e,p} y_{k-5} \right) \\
&\leq \gamma_{a,p} x_{k-1}+\gamma_{b,p} x_{k-3}+\gamma_{c,p} \gamma_{d,p} x_{k-5}+\gamma_{c,p} \gamma_{e,p} y_{k-5} \\
&\leq \gamma_{a,p} x_{k-1}+\gamma_{b,p} x_{k-3}+\gamma_{c,p} \gamma_{d,p} x_{k-5}+\gamma_{c,p} \gamma_{e,p} \left( \gamma_{d,p} x_{k-7}+\gamma_{e,p} y_{k-7} \right) \\
&\leq \gamma_{a,p} x_{k-1}+\gamma_{b,p} x_{k-3}+\gamma_{c,p} \gamma_{d,p} \left( x_{k-5}+\gamma_{e,p} x_{k-7} \right) +\gamma_{c,p} \gamma_{e,p}^{2} y_{k-7} \\
&\leq \gamma_{a,p} x_{k-1}+\gamma_{b,p} x_{k-3}+\gamma_{c,p} \gamma_{d,p} \sum_{i=0}^{\lfloor \frac{k-5}{2} \rfloor} \gamma_{e,p}^{i} x_{k-5-2i}+\gamma_{c,p} \gamma_{e,p}^{\lfloor \frac{k-5}{2} \rfloor +1} y_{k-5-2\lfloor \frac{k-5}{2} \rfloor}.
\end{align*}

\subsubsection{Special case}
If $G_2$ is not invertible in \eqref{eqn:para_err}, we first take the inner product of both sides with $\lambda_p$,
\begin{align*}
\left\{
\begin{aligned}
e_{n+1}^{k+1} &= G_{1}(Jz_{p}) e_{n}^{k+1} + F_{1}(z_{p}) e_{n,-1}^{k} + F_{2}(z_{p}) e_{n}^{k} - G_{2}(Jz_{p}) e_{n}^{k}, \\
e_{n+1,-1}^{k+1} &= G_{2}(Jz_{p}) e_{n}^{k+1} + F_{1}'(z_{p}) e_{n,-1}^{k} + F_{2}'(z_{p}) e_{n}^{k} - G_{2}(Jz_{p}) e_{n}^{k}.
\end{aligned}
\right.
\end{align*}
We consider the case when $G_2(Jz_p)=0$, then the equations above reduce to XXX. \red{also consider G1=0}

\subsection{Comparison} 
The theoretical estimate is almost done, only the last step left. In this section, we compare the convergence rate between the multi-step Parareal by Maday and by our newly proposed update, through \eqref{eqn:bound}. In Maday's work, they choose $G_1 = G_2=R(\Delta TA)$ in \eqref{eqn:m-para}. We propose a more reasonable update: $G_1=R(\Delta TA)$ and $G_2 = R((\Delta T-\Delta t)A)$ in \eqref{eqn:m-para}. Next, we will illustrate these two algorithms through the recursion \eqref{eqn:bound}. 

\subsubsection{$\gamma_{a,p}$}
The expression of the convergence function $\gamma_{a,p}$ is given in \eqref{eqn:g_a}. This can be further bounded by 
\begin{equation*}
\gamma_{a,p} \leq\left( \frac{\gamma_{0,p}}{1-|G_{1}\left( Jz_{p} \right) |} \right)^{2} +\frac{|\bar{F}_{1,p} h_{p}|}{1-|G_{1}\left( Jz_{p} \right) |}.
\end{equation*}
However, this bound is not sharp when $\gamma_{0}^2(z_p)$ and $\bar{F}_1(z_p)h_p(z_p)$ have opposite signs. Then we plot the value of $\gamma_{a,p}$ when $n=1000$ to illustrate the behavior of different updates. In Fig.~\ref{fig:gamma_a_BE}, we plot the graph of $\gamma_a$ when CP is BE. As $J$ increases, the supremum of $\gamma_a$ also becomes larger. In this case, the benefit of the new update can not be observed since BE is not an accuracy solver. For comparison, in Fig.~\ref{fig:gamma_a_Exact}, we plot the graph of $\gamma_a (z)$ when we choose the exact solver as the CP, i.e., $R(z)=\text{exp}(-z)$. Different from the case for BE, as $J$ increases, the supremum of $\gamma_a$ decreases for both updates. Note that the supremum in the new update is much smaller than that one in the Maday's update, which implies that our new update is more reasonable. This can be explained through the important gradient in $\gamma_a$: $\gamma_0$ defined in \eqref{eqn:g_0},
\begin{align*}
    \gamma_{0} \left( z \right) &=F_{1}\left( z \right) +F_{2}\left( z \right) -\text{exp} \left( -Jz \right);&~\text{(Maday)}\\
    \gamma_{0} \left( z \right) &=F_{1}\left( z \right) \text{exp} \left( z \right) +F_{2}\left( z \right) -\text{exp} \left( -Jz \right);&~\text{(New)}.
\end{align*}
Obviously, the new update is more reasonable since $F_1$ and $F_2$ should not be in the same position. 

\begin{figure}[htbp]
    \centering
    \begin{subfigure}[b]{0.48\textwidth}
\includegraphics[width=\linewidth]{Gamma_a_BE_Left.pdf}
    \end{subfigure}
    \begin{subfigure}[b]{0.48\textwidth}
        \includegraphics[width=\linewidth]{Gamma_a_BE_Right.pdf}
    \end{subfigure}
    \caption{The graph of $\gamma_a(z)$ when BE is chosen as the CP. Left: Maday's update; Right: New update.}
    \label{fig:gamma_a_BE}
\end{figure}

\begin{figure}[htbp]
    \centering
    \begin{subfigure}[b]{0.48\textwidth}
\includegraphics[width=\linewidth]{Gamma_a_Exact_Left.pdf}
    \end{subfigure}
    \begin{subfigure}[b]{0.48\textwidth}
        \includegraphics[width=\linewidth]{Gamma_a_Exact_Right.pdf}
    \end{subfigure}
    \caption{The graph of $\gamma_a(z)$ when the exact solver is chosen as the CP. Left: Maday's update; Right: New update.}
    \label{fig:gamma_a_Exact}
\end{figure}

\subsubsection{$\gamma_{b,p}$}
The expression of $\gamma_{b,p}$ is given in \eqref{eqn:g_b}. 
\begin{figure}[htbp]
    \centering
    \begin{subfigure}[b]{0.48\textwidth}
\includegraphics[width=\linewidth]{Gamma_b_BE_Left.pdf}
    \end{subfigure}
    \begin{subfigure}[b]{0.48\textwidth}
        \includegraphics[width=\linewidth]{Gamma_b_BE_Right.pdf}
    \end{subfigure}
    \caption{The graph of $\gamma_b(z)$ when BE is chosen as the CP. Left: Maday's update; Right: New update.}
    \label{fig:gamma_b_BE}
\end{figure}

\begin{figure}[htbp]
    \centering
    \begin{subfigure}[b]{0.48\textwidth}
\includegraphics[width=\linewidth]{Gamma_b_Exact_Left.pdf}
    \end{subfigure}
    \begin{subfigure}[b]{0.48\textwidth}
        \includegraphics[width=\linewidth]{Gamma_b_Exact_Right.pdf}
    \end{subfigure}
    \caption{The graph of $\gamma_b(z)$ when the exact solver is chosen as the CP. Left: Maday's update; Right: New update.}
    \label{fig:gamma_b_Exact}
\end{figure}

\subsubsection{$\gamma_{c,p}\gamma_{d,p}$}
The expression of $\gamma_{c,p}\gamma_{d,p}$ is given in \eqref{eqn:g_c} and \eqref{eqn:g_d,d_e}.
\begin{figure}[htbp]
    \centering
    \begin{subfigure}[b]{0.48\textwidth}
\includegraphics[width=\linewidth]{Gamma_cd_BE_Left.pdf}
    \end{subfigure}
    \begin{subfigure}[b]{0.48\textwidth}
    \includegraphics[width=\linewidth]{Gamma_cd_BE_Right.pdf}
    \end{subfigure}
    \caption{The graph of $\gamma_c(z)\gamma_d(z)$ when BE is chosen as the CP. Left: Maday's update; Right: New update.}
    \label{fig:gamma_cd_BE}
\end{figure}

\begin{figure}[htbp]
    \centering
    \begin{subfigure}[b]{0.48\textwidth}
\includegraphics[width=\linewidth]{Gamma_cd_Exact_Left.pdf}
    \end{subfigure}
    \begin{subfigure}[b]{0.48\textwidth}
        \includegraphics[width=\linewidth]{Gamma_cd_Exact_Right.pdf}
    \end{subfigure}
    \caption{The graph of $\gamma_c(z)\gamma_d(z)$ when the exact solver is chosen as the CP. Left: Maday's update; Right: New update.}
    \label{fig:gamma_cd_Exact}
\end{figure}

\subsubsection{$\gamma_{c,p}\gamma_{d,p} \gamma_{e,p}$}
The expression of $\gamma_{c,p}\gamma_{d,p} \gamma_{e,p}$ is given in \eqref{eqn:g_c} and \eqref{eqn:g_d,d_e}.
\begin{figure}[htbp]
    \centering
    \begin{subfigure}[b]{0.48\textwidth}
\includegraphics[width=\linewidth]{Gamma_cde_BE_Left.pdf}
    \end{subfigure}
    \begin{subfigure}[b]{0.48\textwidth}
        \includegraphics[width=\linewidth]{Gamma_cde_BE_Right.pdf}
    \end{subfigure}
    \caption{The graph of $\gamma_c(z)\gamma_d(z)\gamma_e(z)$ when BE is chosen as the CP. Left: Maday's update; Right: New update.}
    \label{fig:gamma_cde_BE}
\end{figure}

\begin{figure}[htbp]
    \centering
    \begin{subfigure}[b]{0.48\textwidth}
\includegraphics[width=\linewidth]{Gamma_cde_Exact_Left.pdf}
    \end{subfigure}
    \begin{subfigure}[b]{0.48\textwidth}
        \includegraphics[width=\linewidth]{Gamma_cde_Exact_Right.pdf}
    \end{subfigure}
    \caption{The graph of $\gamma_c(z)\gamma_d(z) \gamma_e (z)$ when the exact solver is chosen as the CP. Left: Maday's update; Right: New update.}
    \label{fig:gamma_cde_Exact}
\end{figure}

\bibliographystyle{siamplain}
\bibliography{references}